\documentclass[11pt, reqno]{amsart}

\usepackage[english]{babel}
\usepackage{hyperref}
\usepackage{amsthm}
\usepackage{color}
\usepackage{mathrsfs}
\usepackage{amsmath}
\usepackage{dirtytalk}
\usepackage{mathtools}
\usepackage{amsfonts}
\usepackage{amssymb}
\usepackage{bm}
\usepackage{physics}
\usepackage{enumitem}
\usepackage{tikz-cd}
\usepackage{a4wide}
\usepackage{cleveref}
\usepackage{esint}
\usepackage{nicefrac}
\usepackage{yfonts}
\numberwithin{equation}{section}

\newtheorem{thm}{Theorem}[section]
\newtheorem*{thm*}{Theorem}
\newtheorem{lem}[thm]{Lemma}
\newtheorem{prop}[thm]{Proposition}

\theoremstyle{definition}
\newtheorem{defn}[thm]{Definition}

\theoremstyle{remark}
\newtheorem{rem}[thm]{Remark}

\newcommand{\fr}{\penalty-20\null\hfill$\blacksquare$}

\newcommand{\diff}{{\mathrm{d}}}

\newcommand{\nablatilde}{\bar{\nabla}}
\newcommand{\Prbar}{\bar{\Pr}}
\newcommand{\pibar}{\bar{\pi}}

\newcommand{\qcr}{{\mathrm{QCR}}}

\newcommand{\qcrbar}{{\mathrm{\bar{QCR}}}}
\newcommand{\dive}{{\mathrm{div}}}

\newcommand{\hess}{{\mathrm{Hess}}}
\newcommand{\Hhess}{{\mathbf{Hess}}}

\newcommand{\capa}{{\mathrm {Cap}}}

\newcommand{\Ric}{{\bf {Ric}}}

\newcommand{\TV}{{\mathrm {TV}}}
\newcommand{\Geod}{{\mathrm {Geod}}}

\newcommand{\Ptwo}{{\mathscr{P}_2 }}

\newcommand{\mres}{\mathbin{\vrule height 1.6ex depth 0pt width 0.13ex\vrule height 0.13ex depth 0pt width 1.3ex}}

\newcommand{\dist}{{\mathsf{d}}}
\newcommand{\sfd}{{\dist}}
\newcommand{\DDelta}{{\mathbf{\Delta}}}

\newcommand{\DeltaH}{{\Delta_{\mathrm{H}}}}

\newcommand{\mass}{{\mathsf{m}}}
\newcommand{\mm}{{\mass}}

\newcommand{\XX}{{\mathsf{X}}}
\newcommand{\X}{\XX}

\newcommand{\EE}{{\mathcal{E}}}
\newcommand{\VV}{{\mathcal{V}}}

\newcommand{\Cqcvf}{\mathcal{QC}(T\XX)}

\newcommand{\Ent}{\mathrm{Ent}}

\newcommand{\defeq}{\mathrel{\mathop:}=}

\newcommand{\GG}{\mathcal{G}}
\newcommand{\MM}{\mathcal{M}}
\newcommand{\HG}{\mathrm{H}}

\newcommand{\RR}{\mathbb{R}}
\newcommand{\NN}{\mathbb{N}}

\newcommand{\Ss}{{\mathrm {S}^2}(\XX)}
\newcommand{\Lp}{{\mathrm {L}}}
\newcommand{\Lploc}{\mathrm {L}_\mathrm{loc}}
\newcommand{\Lpo}{\Lp^0(\mass)}
\newcommand{\Lpc}{\Lp^0(\capa)}

\newcommand{\Lpiloc}{\Lp^\infty_{\mathrm{loc}}(\mass)}

\newcommand{\Lpt}{\Lp^2(\mass)}

\newcommand{\Lpi}{\Lp^\infty(\mass)}

\newcommand{\HSloc}{\mathrm {H^{1,2}_{loc}}}

\newcommand{\HSs}{{\mathrm {H^{1,2}(\XX)}}}

\newcommand{\WHCSs}{{\mathrm {H^{1,2}_C}(T\XX)}}
\newcommand{\WSHs}{{\mathrm {W^{1,2}_H}(T\XX)}}
\newcommand{\WHHSs}{{\mathrm {H^{1,2}_H}(T\XX)}}
\newcommand{\HSsloc}{{\mathrm {H^{1,2}_{loc}(\XX)}}}

\newcommand{\LIP}{{\mathrm {LIP}}}

\newcommand{\LIPbs}{{\mathrm {LIP_{bs}}}}

\newcommand{\Cb}{{C_{\mathrm{b}}}}

\newcommand{\LIPb}{{\mathrm {LIP_{b}}}}

\newcommand{\TestF}{{\mathrm {TestF}}}
\newcommand{\TestFbs}{{\mathrm {TestF_{bs}}}}

\newcommand{\TestV}{{\mathrm {TestV}}}

\newcommand{\TestVbar}{{\mathrm {Test\bar{V}}}}

\newcommand{\tanX}{\Lp^2(T\XX)}

\newcommand{\tanXzero}{\Lp^0(T\XX)}
\newcommand{\tanXinf}{\Lp^\infty(T\XX)}

\newcommand{\tanbvXp}[2]{\Lp^{#1}_{#2}(T\XX)}

\newcommand{\tanbvXzero}[1]{\Lp^0_{#1}(T\XX)}

\newcommand{\tanXcap}{\Lp^0_\capa(T\XX)}
\newcommand{\tanXcaptwo}{\Lp^0_\capa(T^{\otimes 2}\XX)}
\newcommand{\tanXcapinf}{\Lp^\infty_\capa(T\XX)}
\newcommand{\RCD}{{\mathrm {RCD}}}

\newcommand{\Meas}{\mathrm{Meas}}

\def\XXint#1#2#3{{\setbox0=\hbox{$#1{#2#3}{\int}$}
		\vcenter{\hbox{$#2#3$}}\kern-.5\wd0}}

\newcommand{\bnabla}{\bm\nabla}
\newcommand{\bsq}[2]{\pmb [#1,#2\pmb]}
\newcommand{\distrlie}[2]{\bsq{#1}{#2}}
\newcommand{\ru}[3]{{\bold R}(#1,#2)(#3)}
\newcommand{\rd}[4]{\bm{\mathcal R}(#1,#2,#3,#4)}

\begin{document}
	\title[Fine representations on RCD spaces]{Fine representation of Hessian of convex functions \\and Ricci tensor on RCD spaces}
	
	\author[C. Brena]{Camillo Brena}
	\address{C.~Brena: Scuola Normale Superiore, Piazza dei Cavalieri 7, 56126 Pisa} 
	\email{\tt camillo.brena@sns.it}
	
	\author[N. Gigli]{Nicola Gigli}
	\address{N.~Gigli: SISSA, Via Bonomea 265, 34136 Trieste} 
	\email{\tt ngigli@sissa.it}
	\begin{abstract}  It is known that on $\RCD$ spaces one can define a distributional Ricci tensor ${\bf Ric}$. Here we give a fine description of this object by showing that it admits the polar decomposition
	$$
	{\bf Ric}=\omega\,|{\bf Ric}|
	$$
	for a suitable non-negative measure $|{\bf Ric}|$ and unitary tensor field $\omega$. The regularity of both the mass measure and of the polar vector are also described. The representation provided here allows to answer some open problems about the structure of the Ricci tensor in such singular setting. 
	
	Our discussion also covers the case of Hessians of convex functions and, under suitable assumptions on the base space, of the Sectional curvature operator.
	\end{abstract}
\maketitle
\tableofcontents
\section*{Introduction}
A classical statement in modern analysis asserts that a positive distribution is a Radon measure. This fact extends to tensor-valued distributions so that, for instance, the distributional Hessian of a convex function on $\RR^d$, that for trivial reasons is a symmetric non-negative matrix-valued distribution, can be represented by a matrix-valued measure. The proof for the tensor-valued case follows from the scalar-valued case simply by looking at the coordinates of the tensor. To put it differently, the fact that on $\RR^d$ we can find an orthonormal base of the tangent bundle made of smooth vectors allows to regard a tensor-valued distribution as a collection of scalar-valued ones and thus to transfer results valid in the latter case into the former one.

\bigskip

The fact that positive functionals defined on a sufficiently large class of functions are represented by measures can be extended far beyond the Euclidean setting, up to at least locally compact spaces: this is the content of the Riesz--Daniell--Stone representation theorems. In this paper we are concerned with the tensor-valued case when the underlying space is an $\RCD(K,N)$ spaces. These class of spaces, introduced in \cite{Gigli12} after \cite{Lott-Villani09}, \cite{Sturm06I,Sturm06II}, \cite{Ambrosio_2014} (see the surveys \cite{AmbICM}, \cite{gigli2023giorgi} and references therein)  are the non-smooth counterpart of Riemannian manifolds with Ricci curvature $\geq K$ and dimension $\leq N$. One of the key features of these spaces, and in fact the essence of the proposal in  \cite{Gigli12}, is that calculus on them is built upon the notions of \say{Sobolev functions} and \say{integration by parts}. As such, it is perhaps not surprising that distribution-like tensors appear frequently in the field. Let us mention three different instances when this occurs, where the relevant tensor is non-negative (or at least bounded from below):
\begin{itemize}
\item[i)] The {\bf Hessian of a convex function}. As observed in \cite{Ketterer_2015,Sturm14}, to a regular enough function $f$ on an $\RCD$  space one can associate a suitable \say{distributional Hessian} that acts on sufficiently smooth vector fields: reformulating a bit the definition in \cite{Ketterer_2015}, the Hessian of $f$ is the map
\begin{align*}
	&\{\text{smooth vector fields}\}\ni  X,Y\\&\qquad\mapsto\quad {\rm Hess}(f)(X,Y)\defeq -\frac{1}{2}\int_\XX \dive\,X\nabla f\,\cdot\, Y+\dive\,Y\nabla f\,\cdot\, X+\nabla f\,\cdot\,\nabla (X\,\cdot\,Y)\dd\mass
\end{align*}
and it turns out, see  \cite[Theorem 7.1]{Ketterer_2015} that under suitable regularity assumptions on $f$ we have
\[
 \text{$f$ is $\kappa$-convex}\qquad\Leftrightarrow\qquad {\rm Hess}(f)(X,X)\geq\kappa\quad\text{for every $X$},
\]
thus matching the Euclidean distributional characterization of convexity.
\item[ii)] The {\bf Ricci curvature} of an $\RCD(K,\infty)$ space. As discussed in \cite{Gigli14}, one can use the Bochner identity to define what the Ricci tensor is in this low regularity setting, by putting 
\begin{align*}
	&\{\text{smooth vector fields}\}\ni  X,Y\\&\qquad\mapsto\quad {\bf Ric}(X,Y)\defeq \DDelta \frac{X\,\cdot\, Y}{2}+\left( \frac{1}{2} X\,\cdot\,\DeltaH Y+ \frac{1}{2} Y\,\cdot\,\DeltaH X-\nabla X\,\cdot\,\nabla Y\right)\mass
\end{align*}

and it turns out that, see  \cite{Gigli14}, in a suitable sense we have
\[
\text{the space $(\X,\sfd,\mm)$ is $\RCD(\kappa,\infty)$}\qquad\Leftrightarrow\qquad \text{${\bf Ric}(X,X)\geq \kappa|X|^2\mm$}\quad\text{for every $X$}.
\]
\item[iii)] The {\bf Sectional curvature} of an $\RCD(K,\infty)$ space. As discussed in \cite{GigliRiemann}, one can give a meaning to the full Riemann curvature tensor on a generic $\RCD$ space. In general, one cannot expect any sort of regularity on it, as the lower bound on the Ricci, encoded in the $\RCD$ assumption, cannot give any information on the full Riemann tensor. Still, this opens up the possibility of saying when is that the \say{sectional curvature of an $\RCD$ space is bounded from below}. The geometric significance of this statement is still unknown.  
\end{itemize}
In each of these cases, a better understanding of the relevant tensor is desirable and a first step in this direction is to comprehend whether the given bound from below forces it to be a measure-like object. To fix the ideas, let us discuss the case of the Ricci curvature: what one would like to know is whether the operator ${\rm Ric}$ described above can be represented via a sort of polar decomposition as
\begin{equation}
\label{eq:reprric}
{\bf Ric}=\omega\,|{\bf Ric}|,
\end{equation}
where $|{\bf Ric}|$ is a non-negative measure and $\omega$ is a tensor of norm 1 $|{\bf Ric}|$-a.e., meaning that the identity
\[
{\bf Ric}(X,Y)=\omega\cdot(X\otimes Y)\,|{\bf Ric}|
\]
holds as measures for any couple of sufficiently smooth vector fields $X,Y$. The main result of this manuscript is that, yes, a representation like \eqref{eq:reprric} holds for the three tensors discussed above.

Few important remarks are in order (we shall discuss the case of the Ricci curvature, but similar comments are in place for the Hessian and the sectional curvature):
\begin{itemize}
\item[-] A writing like that in the right hand side of \eqref{eq:reprric} requires the tensor field $\omega$ to be $|{\bf Ric}|$-well-defined. In this respect notice that on one side on $\RCD$ spaces tensor fields can be well-defined up to $\capa$-null sets, where $\capa$ is the 2-capacity (in some sense, thanks to the fact that one can speak about Sobolev vector fields - see \cite{debin2019quasicontinuous}). On the other hand,    the mass measure $|{\bf Ric}|$ is absolutely continuous with respect to $\capa$ (because the distributional definition of Ricci tensor  is continuous on the space of Sobolev vector fields). The combination of these two facts makes it possible the writing in \eqref{eq:reprric}. This perfect matching between the regularity achievable by $\omega$ and the sets that can actually be charged by $|{\bf Ric}|$ is far from being a coincidence.

\item[-] The construction of the polar decomposition as in \eqref{eq:reprric}  follows the same rough idea described at the beginning of the introduction: we would like to take a pointwise orthonormal base $X_1,\ldots,X_n$ of sufficiently regular vector field and then study the real valued functionals $\varphi\mapsto \int_\XX\varphi\,\dd {\bf Ric}(X_i,X_j)$. Clearly, even in a smooth Riemannian manifold one in general cannot find such a global orthonormal base, but a first problem we encounter here is that such bases only exists on suitable Borel sets $A_k\subseteq\X$ (whose interior might in general be empty). This causes severe technical complications  in handling the necessary localization arguments, see for instance the proof of Theorem \ref{hessrepr}.

\item[-] Related to the above there is the fact that the mass measure $|{\bf Ric}|$ turns out to be a $\sigma$-finite Borel measure that in general is \emph{not} Radon. More precisely, on the sets $A_k\subseteq\X$ on which we have a pointwise orthonormal base, the restriction of $|{\bf Ric}|$ is finite (whence $\sigma$-finiteness). However, in general it might very well be that there is some point $x\in\X$ such that every neighbourhood of $x$ encounters infinitely many of the $A_k$'s. This happens even in very simple examples such that the tip of a cone, as it is known, see \cite{DePhilZim} and reference therein, that at the tip of a cone every sufficiently regular vector field must vanish.

\item[-] Despite the above, for any couple of sufficiently regular vector fields $X,Y$ we have   $w\cdot( X\otimes Y)\in L^1(|{\bf Ric}|)$.

\item[-] Since $|{\bf Ric}|$ is not Radon, in constructing the representation \eqref{eq:reprric}, and more generally in understanding these distributional objects we have discussed,  we cannot rely on the theory of local vector measures that we recently developed in \cite{BGlvm}.

\item[-] The representation  \eqref{eq:reprric} marks a clear step forward in the understanding of the Ricci tensor on $\RCD$ spaces, as what was previously manageable only via integration by parts - and thus required regularity of the vector fields involved - now is realized as a 0th-order object and thus has a more pointwise meaning. For instance, it allows to quickly solve a problem that was left open in \cite{Gigli14}. The problem was as follows: suppose that $X_i,X,Y$, $i=1,\ldots,n$, are smooth vector fields, that $f_i\in C_b(\X)$ and that $\sum_if_iX_i=X$. Can we conclude that $\sum_if_i{\bf Ric}(X_i,Y)={\bf Ric}(X,Y)$? One certainly expects the answer to be affirmative, but if the only definition of ${\bf Ric}$ involves integration by parts - as it was the case in \cite{Gigli14} - then it is not clear how to conclude, given that in general $f_iX_i$ is not regular enough to justify the necessary computations. On the other hand, the representation  \eqref{eq:reprric} immediately allows to positively answer the question.
\end{itemize}

\section*{Acknowledgements}
The authors wish to thank Luigi Ambrosio for useful suggestions.
Part of this work was carried out at the Fields Institute during the Thematic Program on Nonsmooth Riemannian and Lorentzian Geometry, Toronto 2022. The authors gratefully acknowledge the warm hospitality and the stimulating atmosphere.
\section{Preliminaries}

\subsection{RCD spaces}
In this note we are going to consider $\RCD$ spaces, which we now briefly introduce. An $\RCD(K,N)$ space is an infinitesimally Hilbertian (\cite{Gigli12}) metric measure space $(\XX,\dist,\mass)$  satisfying a lower Ricci curvature bound and an upper dimension bound (meaningful if $N<\infty$) in a synthetic sense according to \cite{Sturm06I,Sturm06II,Lott-Villani09}, see \cite{AmbICM,Villani2017,gigli2023giorgi} and references therein. We assume the reader to be familiar with this material. Whenever we write $\RCD(K,N)$, we implicitly assume that $N<\infty$, unless otherwise stated.

Also, we assume that the reader is familiar with the calculus developed on this kind of non-smooth structures (\cite{Gigli12,Gigli14}, see also \cite{Gigli17,GP19}): in particular, we assume familiarity with Sobolev spaces (and heat flow), with the notions of  (co)tangent module and its tensor and exterior products (see also Section \ref{subsect1} and \ref{subsect2}), and with the  notions of divergence, Laplacian, Hessian and covariant derivative, together with their properties. 

\medskip

We give now our working definition for the space of test functions and test vector fields.
Following \cite{Gigli14,Savare13} (with the additional request of a $\Lp^\infty$ bound on the Laplacian), we define the vector space of test functions on an $\RCD(K,\infty)$ space as
\begin{equation}\notag
	\TestF(\XX)\defeq\{f\in\LIP(\RR)\cap\Lpi\cap D(\Delta): 	\Delta f\in \HSs\cap\Lpi\},
\end{equation}
and the vector space of test vector fields as
\begin{equation}\notag
	\TestV(\XX)\defeq\left\{ \sum_{i=1}^n f_i\nabla g_i : f_i\in\Ss\cap\Lpi,g_i\in\TestF(\XX)\right\}.
\end{equation}
To be precise, the original definition of $\TestV(\XX)$ given by the second named author was slightly different. However, when using test vector fields to define regular subsets of vector fields such as $\WHCSs$ and $\WHHSs$, the two definitions produce the same subspaces, see for example the proofs of \cite[Lemma 3.3 and Lemma 3.4]{BGBV}. The advantage of working with this slightly more general class lies in the fact that 
$$
\frac{1}{1\vee |v|}v\in\TestV(\XX)\qquad\text{for every }v\in\TestV(\XX),
$$
whereas the drawback is that for $v\in\TestV(\XX)$, in general we do not have $\dive(v)\in\Lpi$. Nevertheless, we are still going to need the classical definition  of test vector fields (used, in particular in the references \cite{Gigli14,GigliRiemann} for what concerns Ricci and Riemann tensors): we call such space $\VV$, i.e.\
\begin{equation}\label{oldtest}
	\VV\defeq	\left\{ \sum_{i=1}^n f_i\nabla g_i : f_i,g_i\in\TestF(\XX)\right\}.
\end{equation}
\medskip

For future reference, we recall here \cite[Lemma 3.3]{BGBV}. 
	\begin{lem}\label{calculushodge}
	Let $(\XX,\dist,\mass)$ be an $\RCD(K,\infty)$ space, $X\in\WSHs\cap\tanXinf$ and $f\in \mathrm{S}^2(\XX)\cap\Lpi$. Then $f X\in \WSHs$ and 
	\begin{alignat}{2}\notag
		\dive (f X)&=\nabla f\,\cdot\, X+f\dive\, X,\\\notag
		\diff (f X)&=\nabla f\wedge X+f\diff X.
	\end{alignat}
	If moreover $X\in\WHHSs$, then $f X\in \WHHSs$.
\end{lem}

The following calculus lemma will serve as a key tool in proving, in a certain sense, a strong locality property of some measures.
\begin{lem}\label{Amb}
	Let $(\XX,\dist,\mass)$ be an $\RCD(K,\infty)$ space and let $X\in\WHHSs$.
	Take $\{\tilde{\varphi}_n\}_n\subseteq\LIPb(\RR)$ 
	defined by
	\begin{equation}\notag
		\tilde{\varphi}_n(x)\defeq
		\begin{cases}
			1\qquad&\text{if }x\le 0,\\
			1-n x\qquad&\text{if } 0<x<n^{-1},\\
			0\qquad&\text{if }x\ge n^{-1} .\\
		\end{cases}
	\end{equation}
	Let then ${\varphi}_n\defeq\tilde{\varphi}_n\circ \abs{X}$. Then $\varphi_n\in\HSs$, $\Vert \varphi_n\Vert_{\Lpi}\le 1$ and $\varphi_n X\rightarrow 0$ in the $\WSHs$ topology.
\end{lem}
\begin{proof}
	By \cite[Lemma 2.5]{debin2019quasicontinuous}, $\varphi_n\in\HSs$ for every $n$ with 
	$$
	\abs{\nabla \varphi_n}= |{\tilde{\varphi}_n'}|\circ\abs{X}\,\abs{\nabla \abs{X}}\le n\chi_{\{\abs{X}\in(0,n^{-1})\}}\abs{\nabla X}\qquad\mass\text{-a.e.}
	$$
	In particular,
	\begin{equation}\label{to0}
		\abs {\nabla \varphi_n}\abs{X}\le \chi_{\{\abs{X}\in(0,n^{-1})\}}\abs{\nabla X}\qquad\mass\text{-a.e.}
	\end{equation}
	
	Notice that integrating by parts and using standard approximation arguments, taking into account \eqref{to0} (which also gives the membership of the right hand sides to the relevant spaces) we have
	\begin{align*}
		\dive(\varphi_n X)&=\nabla\varphi_n\,\cdot\,X+\varphi_n \dive\,X\in\Lpt\\	\diff(\varphi_n X)&=\nabla\varphi_n\wedge X+\varphi_n \dd\,X\in\mathrm{L}^{2}(\Lambda^2 T^*\XX).
	\end{align*}
	Therefore,
	\begin{align*}
		\dive(\varphi_n X)\rightarrow \chi_{\{\abs{X}=0\}}\dive\, X\qquad&\text{in }\Lpt\\	\diff(\varphi_n X)\rightarrow \chi_{\{\abs{X}=0\}}\diff X\qquad&\text{in }\mathrm{L}^{2}(\Lambda^2 T^*\XX).
	\end{align*}
	Now, by dominated convergence, $\varphi_n X\rightarrow 0$ in $\tanX$ so that by the closure of the operators $\dive$ and $\diff$ (see \cite[Theorem 3.5.2]{Gigli14}), it holds that $\varphi_n X\rightarrow 0$ in $\WSHs$.
\end{proof}

\begin{rem}
	Inspecting the proof of Lemma \ref{Amb}, we see that if $(\XX,\dist,\mass)$ is an $\RCD(K,\infty)$ space and $X\in\WHHSs$, then $\dive\,X=0\ \mass$-a.e.\ on $\{X=0\}$ and similarly $\diff X=0\ \mass$-a.e.\ on $\{X=0\}$.\fr
\end{rem}

\subsection{Cap-modules}\label{cdnocads}
In this subsection we recall the basic theory of $\capa$-modules for $\RCD$ spaces. 
We assume familiarity with the definition of capacitary modules, quasi-continuous functions and vector fields and related material in \cite{debin2019quasicontinuous}. A summary of the material we will use can be found in \cite[Section 1.3]{bru2019rectifiability}. For the reader's convenience, we write the results that we will need most frequently.
\medskip

First, we recall that exploiting Sobolev functions, we define the $2$-capacity (to which we shall simply refer as capacity) of any set $A\subseteq\XX$ as 
\begin{equation}\label{defcapa}
	\capa(A)\defeq\inf\left\{ \Vert f\Vert_{\HSs}^2:f\in\HSs,\ f\ge 1\ \mass\text{-a.e.\ on some neighbourhood of $A$}\right\}. 
\end{equation}
An important object will be the one of fine tangent module, as follows ($\qcr$ stands for \say{quasi continuous representative}).

\begin{thm}[{\cite[Theorem 2.6]{debin2019quasicontinuous}}]\label{tancapa}
	Let $(\XX,\dist,\mass)$ be an $\RCD(K,\infty)$ space.
	Then there exists a unique couple $(\tanXcap,\nablatilde)$, where $\tanXcap$ is a $\Lp^0(\capa)$-normed $\Lp^0(\capa)$-module and $\nablatilde:	\TestF(\XX) \rightarrow\tanXcap$ is a linear operator such that:
	\begin{enumerate}[label=\roman*)]
		\item
		$|{\nablatilde f}|=\qcr(\abs{\nabla f}) \ \capa$-a.e.\ for every $f\in\TestF(\XX)$,
		\item
		the set $\left\{\sum_{n} \chi_{E_n}\nablatilde f_n\right\}$, where $\{f_n\}_n\subseteq\TestF(\XX)$ and $\{E_n\}_n$ is a Borel partition of $\XX$ is dense in $\tanXcap$.
	\end{enumerate}
	Uniqueness is intended up to unique isomorphism, this is to say that if another couple $(\tanXcap',\nablatilde')$ satisfies the same properties, then there exists a unique module isomorphism $\Phi:\tanXcap\rightarrow\tanXcap'$ such that $\Phi\circ \nablatilde=\nablatilde'$.
	Moreover, $\tanXcap$ is a Hilbert module that we call capacitary tangent module.
\end{thm}

Notice that we can, and will, extend the map $\qcr$ (\cite{debin2019quasicontinuous}) from $\HSs$ to $\Ss\cap\Lpi$ by a locality argument. Also, we often omit to write the map $\qcr$.
We define
\begin{equation}\notag
	\TestVbar(\XX)\defeq \left\{ \sum_{i=1}^n \qcr(f_i) \nablatilde g_i :f_i\in\Ss\cap\Lpi,g_i\in\TestF(\XX) \right\}\subseteq\tanXcap.
\end{equation}

We define also the vector subspace of quasi-continuous vector fields, $\Cqcvf$, as the closure of $\TestVbar(\XX)$ in $\tanXcap$.

Recall now that as $\mass\ll\capa$, we have a natural projection map
\begin{equation}\notag
	\Pr:\Lpc\rightarrow\Lpo \qquad \text{defined as}\qquad[f]_{\Lpc}\mapsto [f]_{\Lpo}
\end{equation}
where $[f]_{\Lpc}$ (resp.\ $[f]_{\Lpo}$) denotes the $\capa$ (resp.\ $\mass$) equivalence class of $f$. It turns out that $\Pr$, restricted to the set of quasi-continuous functions, is injective (\cite[Proposition 1.18]{debin2019quasicontinuous}).
We have the following projection map $\Prbar$, given by \cite[Proposition 2.9 and Proposition 2.13]{debin2019quasicontinuous}, that plays the role of $\Pr$ on vector fields. 
\begin{prop}\label{prbardef}
	Let $(\XX,\dist,\mass)$ be an $\RCD(K,\infty)$ space. There exists a unique linear continuous map \begin{equation}\notag
		\Prbar :\tanXcap\rightarrow\tanXzero
	\end{equation}
	that satisfies
	\begin{enumerate}[label=\roman*)]
		\item $\Prbar (\nablatilde f)=\nabla f$ for every $f\in\TestF(\XX)$,
		\item $\Prbar (g v)=\Pr(g)\Prbar(v)$ for every $g\in\Lpc$ and $v\in\tanXcap$.
	\end{enumerate}
	Moreover, for every $v\in\tanXcap$,
	\begin{equation}\notag
		\abs{\Prbar(v)}=\Pr(\abs{v})\qquad\mass\text{-a.e.}
	\end{equation}
	and $\Prbar$, when restricted to the set of quasi-continuous vector fields, is injective.
\end{prop}
Notice that $\Prbar(\TestVbar(\XX))=\TestV(\XX)$. When there is be no ambiguity, we omit to write the map $\Prbar$.

\begin{thm}[{\cite[Theorem 2.14 and Proposition 2.13]{debin2019quasicontinuous}}]
	Let $(\XX,\dist,\mass)$ be an $\RCD(K,\infty)$ space. Then there exists a unique map $\qcrbar:\WHCSs\rightarrow\tanXcap$ such that
	\begin{enumerate}[label=\roman*)]
		\item $\qcrbar (v)\in\Cqcvf$ for every $v\in\WHCSs$,
		\item $\Prbar\circ{\qcrbar}(v)=v$ for every $v\in\WHCSs$.
	\end{enumerate}
	Moreover, $\qcrbar$ is linear and satisfies
	\begin{equation}\notag
		\abs{\qcrbar(v)}=\qcr(\abs{v})\qquad \capa\text{-a.e.\ for every }v\in\WHCSs,
	\end{equation}
	so that $\qcrbar$ is continuous.
\end{thm}

We will often omit to write the $\qcrbar$ operator for simplicity of notation (but it will be clear from the context when we need the fine representative). This should cause no ambiguity thanks to the fact that 
\begin{equation}\label{qcrfactorizes}
	\qcrbar(g v)=\qcr(g)\qcrbar(v) \qquad\text{for every }g\in\Ss\cap\Lpi\text{ and } v\in\WHCSs\cap\tanXinf.
\end{equation}
This can be proved easily by locality and using the fact that the continuity of the map $\qcr$ implies that $\qcr(g) \qcrbar(v)$ as above is quasi-continuous and the injectivity of the map $\Prbar$ restricted the set of quasi-continuous vector fields yields the conclusion.

The following theorem, that is {\cite[Section 1.3]{bru2019rectifiability}}, will be crucial in the construction of modules tailored to particular measures (see \cite[Theorem 3.10]{BGBV} for an explicit proof of this result).
\begin{thm}
	\label{finemodule}
	Let $(\XX,\dist,\mass)$ be a metric measure space and let $\mu$ be a Borel measure finite on balls such that $\mu\ll\capa$. Let also $\MM$ be a $\Lpc$-normed $\Lpc$-module. Define the natural (continuous) projection 
	\begin{equation}\notag
		\pi_\mu:\Lpc\rightarrow\Lp^0(\mu).
	\end{equation}
	We define an equivalence relation $\sim_\mu$ on $\MM$ as 
	\begin{equation}\notag
		v\sim_\mu w \text{ if and only if } \abs{v-w}=0 \qquad \mu\text{-a.e.}
	\end{equation} 
	Define the quotient module $\MM_{\mu}^0\defeq{\MM}/{\sim_\mu}$ with the natural  (continuous) projection
	\begin{equation}\notag
		\pibar_\mu:\MM\rightarrow\MM_{\mu}^0.
	\end{equation}
	Then $\MM_{\mu}^0$ is a $\Lp^0(\mu)$-normed $\Lp^0(\mu)$-module, with the pointwise norm and product induced by the ones of $\MM$: more precisely, for every $v\in\MM$ and $g\in\Lpc$,
	\begin{equation}\label{defnproj}
		\begin{cases}
			\abs{\pibar_\mu(v)}\defeq\pi_\mu(\abs{v}),\\
			\pi_\mu(g)\pibar_\mu(v)\defeq\pibar_\mu( g v).
		\end{cases}
	\end{equation}
	
	If $p\in[1,\infty]$, we set
	\begin{equation}\notag
		\MM^{p}_{\mu}\defeq\left\{ v\in\MM^0_{\mu}:\abs{v}\in\Lp^p(\mu)\right\},
	\end{equation}
	that is a $\Lp^p(\mu)$-normed $\Lp^\infty(\mu)$-module.
	Moreover, if $\MM$ is a Hilbert module, also $\MM_\mu^0$ and $\MM_\mu^2$ are Hilbert modules. 
\end{thm}
Similarly as for $\qcrbar$, we often omit to write the $\pibar_\mu$ operator (and also the $\pi_\mu$ operator) for simplicity of notation (but it will be clear from the context when we need the fine representative). Again, this should cause no ambiguity thanks to \eqref{qcrfactorizes} and \eqref{defnproj}.

The following lemma, that is  \cite[Lemma 2.7]{bru2019rectifiability}, provides us with the density of test vector fields in quotient tangent modules.
\begin{lem}\label{densitytestbargen}
	Let $(\XX,\dist,\mass)$ be an $\RCD(K,\infty)$ space and let $\mu$ be a finite Borel measure such that $\mu\ll\capa$. Then $\TestV_\mu(\XX)$ is dense in $\tanbvXp{p}{\mu}$ for every $p\in[1,\infty)$.
\end{lem}

In what follows, with a little abuse, we often write, for $v\in\tanXcap$, $v\in D(\dive)$ if and only if $\Prbar (v)\in D(\dive)$ and, if this is the case, $\dive\, v=\dive(\Prbar(v))$. Similar notation will be used for other operators acting on subspaces of $\tanXzero$. 

\medskip

\subsubsection{Tensor product}\label{subsect1}
In this subsection we study the tensor product of normed modules. We focus on $\capa$-modules, as in these spaces we are going to find the main objects of this note. Fix then $n\in\NN$, $n\ge 1$.  We assume that $(\XX,\dist,\mass)$ is an $\RCD(K,\infty)$ space, even though this is clearly not always needed.
Just for the sake of notation, we set $$\mathrm{L}^2(T^{\otimes n}\XX)\defeq\tanX^{\otimes n}.$$

Let now $\mathcal C\subseteq\tanX$ be a subspace. We define $\mathcal C^{\otimes n}$ as the vector subspace of the Hilbert $\Lpt$-normed $\Lpi$-module $\mathrm{L}^2(T^{\otimes n}\XX)$ that consists of finite sums of decomposable tensors of the type $v_1\otimes\cdots\otimes v_n$ where $v_1,\dots,v_n\in \mathcal C$, endowed with the structure of module (included the pointwise norm) induced by the one of $\mathrm{L}^2(T^{\otimes n}\XX)$.
Notice that we can equivalently define $\mathcal C^{\otimes n}$ as follows. First, we consider the multilinear map $$(v_1,\dots,v_n)\mapsto v_1\otimes\cdots\otimes v_n \in\mathrm{L}^2(T^{\otimes n}\XX)\qquad\text{if }v_1\dots,v_n\in\mathcal C$$
that factorizes to a well defined linear map
$$\mathcal C\otimes^{\mathrm {alg}}_\RR\cdots\otimes^{\mathrm {alg}}_\RR\mathcal C\rightarrow \mathrm{L}^2(T^{\otimes n}\XX)$$
and see that $\mathcal C^{\otimes n}$ coincides with the image of this map.
Notice that, unless we are in pathological cases, the map we have just defined is not injective. This is equivalent to the fact that not every map defined on $\mathcal C\otimes^{\mathrm {alg}}_\RR\cdots\otimes^{\mathrm {alg}}_\RR\mathcal C$ induces a map defined on $\mathcal C^{\otimes n}$, in general.

\medskip

Let now $\MM$ be an Hilbert $\Lpc$-normed $\Lpc$-module. We define the $\Lpc$-normed $\Lpc$-module $\MM^{\otimes n}$ repeating the construction done to define the tensor product of $\Lp^0(\mass)$-normed $\Lp^0(\mass)$-modules in \cite[Subsection 3.2.2]{GP19} (originally of \cite{Gigli14}), that is endowing the algebraic tensor product $$\MM\otimes^\textrm{alg}_{\Lpc}\cdots\otimes^\textrm{alg}_{\Lpc}\MM$$ with the pointwise Hilbert-Schmidt norm and then taking the completion with respect to the induced distance.

If $\mu$ is a Borel measure finite on balls and such that $\mu\ll\capa$,  we set $$\mathrm{L}_\mu^p(T^{\otimes n}\XX)\defeq \tanbvXp{p}{\mu}^{\otimes n}\qquad\text{for }p\in\{0\}\cup[1,\infty],$$
where the right hand side is given by Theorem \ref{finemodule}.


\begin{rem}\label{remcompatibility}
	Let $\mu$ be a Borel measure, finite on balls, such that $\mu\ll\capa$. Let also $\MM$ be an Hilbert $\Lpc$-normed $\Lpc$-module.
	Then, using the notation of Theorem \ref{finemodule}, we have a canonical isomorphism
	\begin{equation}\notag
		(\MM^{\otimes n})^0_\mu\cong(\MM^0_\mu)^{\otimes n}.
	\end{equation} 
	This isomorphism is obtained using the map induced by the well defined multilinear map
	$$(\MM^0_\mu)^n\ni ([v_1]_{\sim_\mu},\dots, [v_n]_{\sim_\mu})\mapsto [(v_1\otimes\dots\otimes v_n)]_{\sim_\mu}\in(\MM^{\otimes n})^0_\mu$$
	and noticing that such map turns out to be an isometry with dense image between complete spaces.
	Therefore we also have the canonical inclusion
	\begin{equation}\notag
		(\MM^{\otimes n})^p_\mu\cong\left\{v\in (\MM^0_\mu)^{\otimes n}:\abs{v}\in\Lp^p(\mu)\right\}\qquad\text{if }p\in[1,\infty].
	\end{equation} 
	In particular, with the obvious interpretation for 	$\mathrm{L}^0_\capa(T^{\otimes n}\XX)$,
	\begin{equation}\notag
		(\mathrm{L}^0_\capa(T^{\otimes n}\XX))^0_\mu\cong
		\mathrm{L}_\mu^0(T^{\otimes n}\XX)
	\end{equation} 
	so that
	\begin{equation}\notag
		(\mathrm{L}_\capa^0(T^{\otimes n}\XX))^p_\mu\cong\mathrm{L}_\mu^p(T^{\otimes n}\XX)\qquad\text{if }p\in[1,\infty],
	\end{equation} 
	where $\mathrm{L}_\mu^p(T^{\otimes n}\XX)\defeq\left\{v\in \mathrm{L}_\mu^0(T^{\otimes n}\XX):\abs{v}\in\Lp^p(\mu)\right\}$.\fr
\end{rem}

\medskip

Now we consider $\TestV(\XX)^{\otimes n}$. Notice that $\TestV(\XX)^{\otimes n}$  is a module over the ring $\Ss\cap\Lpi$ in the algebraic sense and, by Lemma \ref{normsobo} below, 
\begin{equation}\notag
	\frac{1}{1\vee|v|}v\in\TestV(\XX)^{\otimes n}\qquad\text{for every }v\in\TestV(\XX)^{\otimes n}.
\end{equation}
By the following lemma (with $\mu=\mass$), $\TestV(\XX)^{\otimes n}$ is dense in ${\rm L}^2(T^{\otimes n}\XX)$, in particular, $\TestV(\XX)^{\otimes n}$ generates in the sense of modules ${\rm L}^2(T^{\otimes2}\XX)$. The following lemma is proved with an approximation argument as in \cite[Lemma 3.2.21]{GP19}, and is a generalization of Lemma \ref{densitytestbargen}.
\begin{lem}\label{densitytestbargentens}
	Let $(\XX,\dist,\mass)$ be an $\RCD(K,\infty)$ space and let $\mu$ be a finite Borel measure such that $\mu\ll\capa$. Then
	$\TestV(\XX)^{\otimes n}$
	is dense in $\mathrm{L}_\mu^p(T^{\otimes n}\XX)$ for every $p\in[1,\infty)$.
\end{lem}
\begin{proof}
	Fix $p\in[1,\infty)$ and $w\in\mathrm{L}_\mu^p(T^{\otimes n}\XX)$. This is to say that $\abs{w}\in\Lp^p(\mu)$ and we can find a sequence of tensors $$\{w^k\}_k\subseteq \tanbvXzero{\mu}\otimes^\textrm{alg}_{\Lp^0(\mu)}\cdots\otimes^\textrm{alg}_{\Lp^0(\mu)}\tanbvXzero{\mu}$$ such that $\abs{w-w^k}\rightarrow 0$ in $\Lp^0(\mu)$. We can assume that $\{|{w^k}|\}_k\subseteq\Lp^p(\mu)$ is a bounded sequence, up to replacing $w^k$ with $\frac{\abs{w}}{|{w^k}|}w^k$. In this case, by dominated convergence, we see that $w^k\rightarrow w$ in $\mathrm{L}_\mu^p(T^{\otimes n}\XX)$. As a consequence of this discussion, an orthonormalization procedure and a truncation argument, we see that we can reduce ourselves to the case $w=w_1\otimes\cdots \otimes w_n$ where $w_i\in\mathrm{L}_\mu^\infty(T^{\otimes n}\XX)$ for every $i$. Then we can conclude iterating Lemma \ref{densitytestbargen}.
\end{proof}

\begin{lem}\label{normsobo}
	Let $(\XX,\dist,\mass)$ be an $\RCD(K,\infty)$ space and let $v\in(\WHCSs\cap\Lpi)^{\otimes n}$. Then $\abs{v}\in\HSs\cap\Lpi$.
\end{lem} 
\begin{proof}
	Fix $v=\sum_{i=1}^m v_1^i\otimes\dots\otimes v_n^i\in(\WHCSs\cap\Lpi)^{\otimes n}$, where $\{v_j^i\}\subseteq\WHCSs\cap\Lpi$, so that there exist $H_v\in\RR$ and $h_v\in\Lpt$ such that for every $i=1,\dots,m$ and $j=1,\dots,n$ it holds \begin{alignat}{5}\notag
		&\abs{v_j^i}\le H_v\qquad&&\mass\text{-a.e.}\\\notag
		&\abs{\nabla v_j^i}\le h_v \qquad&&\mass\text{-a.e.}\notag
	\end{alignat}
	Following a standard argument as e.g.\ in the proof of \cite[Lemma 2.5]{debin2019quasicontinuous} it is enough to show that $|\nabla \abs{v}^2|\le g_v \abs{v}\ \mass$-a.e.\ for some $g_v\in\Lpt$. 
	For  $Z\in\tanXzero$ we define $\nabla_Z v$ as in the discussion right above \cite[Proposition 3.4.6]{Gigli14}, that is the unique vector field in $\tanXzero$ such that for every $Y\in\tanXzero$
	$$
	\nabla_Z v\,\cdot\,Y=\nabla v\,\cdot\, Z\otimes Y\qquad\mass\text{-a.e.}
	$$
	Clearly, $\abs{\nabla_Z v}\le \abs{\nabla v}\abs{Z}$.
	If $Z\in\tanXzero$, we compute, by \cite[Proposition 3.4.6 $i)$]{Gigli14},
	\begin{equation}\notag
		\begin{split}
			&\nabla(v\,\cdot\, v)\,\cdot\,Z=\sum_{i,j=1}^m \nabla\left(\prod_{k=1}^n v_k^i\,\cdot\,v_k^j\right)\,\cdot\,Z=2 \sum_{i,j=1}^m\sum_{h=1}^n  (\nabla_Z v_h^i)\,\cdot\,v_h^j \prod_{\substack{k=1\\ k\ne h}}^n v_k^i\,\cdot\,v_k^j
			\\
			&\qquad=2 \left( \sum_{i=1}^m\sum_{h=1}^n v^i_1\otimes\cdots\otimes v_{h-1}^i\otimes \nabla_Z v_h^i\otimes v_{h+1}^i\otimes\cdots\otimes v^i_n\right)\,\cdot\,\left(\sum_{j=1}^m v^j_1\otimes\cdots\otimes v_n^j\right).
		\end{split}
	\end{equation}
	Now we have finished, as, by the arbitrariness of $Z\in\tanXzero$, the equality above implies that
	\begin{equation*}
		|\nabla \abs{v}^2|\le 2 m n H_v^{n-1} h_v \abs{v}\qquad \mass\text{-a.e.}\qedhere
	\end{equation*}
\end{proof}

\medskip 

We consider now the multilinear map\begin{equation}\label{map}
	(\WHCSs\cap\Lpi)^n\ni(v_1,\dots,v_n)\mapsto \qcrbar(v_1)\otimes\cdots\otimes\qcrbar(v_n)\in\mathrm{L}_\capa^0(T^{\otimes n}\XX)
\end{equation}
and we notice that  (the left hand side is well defined thanks to  Lemma \ref{normsobo} - but as there is a squared norm here, this is indeed trivial)  \begin{equation}\label{equality}
	\qcr\left(\abs{ \sum_{i=1}^m v^i_1\otimes\cdots\otimes v^i_n}^2\right)=\abs{	\sum_{i=1}^m\qcrbar(v^i_1)\otimes\cdots\otimes\qcrbar(v^i_n)}^2 \qquad\capa\text{-a.e.}
\end{equation} so that the map in \eqref{map}
induces a map $\qcrbar:(\WHCSs\cap\Lpi)^{\otimes n}\rightarrow\mathrm{L}_\capa^0(T^{\otimes n}\XX)$ that satisfies, thanks to \eqref{equality},
\begin{equation}\notag
	\qcr({\abs{v}})=\abs{\qcrbar(v)} \qquad\capa\text{-a.e.\ for every }v\in (\WHCSs\cap\Lpi)^{\otimes n}.
\end{equation}
As usual, we often omit to write the maps $\qcrbar$ and $\qcr$.

\medskip

\subsubsection{Exterior  power}\label{subsect2}
Much like the previous section dealt with tensor product, in this section we deal with exterior power (see \cite{Gigli14}). As the arguments are mostly identical, we are going just to sketch the key ideas and we will assume familiarity with the related theory. We assume again that $(\XX,\dist,\mass)$ is a $\RCD(K,\infty)$ space.
Similarly to what done before, we set
$$
\mathrm{L}^2(T^{\wedge n}\XX)\defeq\tanX^{\wedge n}
$$
and define, for $\mathcal C\subseteq \tanX$, $\mathcal C^{\wedge n}$ as before.

\medskip

Let now $\MM$ be an Hilbert $\Lpc$-normed $\Lpc$-module. We define the $\Lpc$-normed $\Lpc$-module $\MM^{\wedge n}$ as the quotient of $
\MM^{\wedge n}$ with respect to the closure of the subspace generated by elements of the form $v_1\otimes\dots\otimes v_n$, where $v_1,\dots, v_n\in\MM$ are such that $v_i=v_j$ for at least two different indices $i$ and $j$. This definition is the trivial adaptation of \cite{Gigli14} to our context and it is possible to prove that $\MM^{\wedge n}$  is indeed an $\Lpc$-normed $\Lpc$-module and that its scalar product is characterized by
$$
v_1\wedge\cdots\wedge v_n\,\cdot\,w_1\wedge\cdots\wedge w_n\,=\det(v_i\,\cdot\,w_j)\qquad\capa\text{-a.e.}$$

Similarly to what done before, if $\mu$ is a Borel measure finite on balls such that $\mu\ll\capa$, we set
$$\mathrm{L}_\mu^p(T^{\wedge n}\XX)\defeq \tanbvXp{p}{\mu}^{\wedge n}\qquad\text{for }p\in\{0\}\cup[1,\infty].$$
\begin{rem}
	Now we adapt Remark \ref{remcompatibility}.
		Let $\mu$ be a Borel measure, finite on balls, such that $\mu\ll\capa$. Let also $\MM$ be an Hilbert $\Lpc$-normed $\Lpc$-module.
	Then, using the notation of Theorem \ref{finemodule}, we have a canonical isomorphism
	\begin{equation}\notag
	(\MM^{\wedge n})^0_\mu\cong(\MM^0_\mu)^{\wedge n}.
\end{equation} 
	This isomorphism is obtained using the map induced by the well defined multilinear map
$$(\MM^0_\mu)^n\ni ([v_1]_{\sim_\mu},\dots, [v_n]_{\sim_\mu})\mapsto [(v_1\wedge\dots\wedge v_n)]_{\sim_\mu}\in(\MM^{\wedge n})^0_\mu$$
and noticing that such map turns out to be an isometry with dense image between complete spaces.
	Therefore we also have the canonical inclusion
\begin{equation}\notag
	(\MM^{\wedge n})^p_\mu\cong\left\{v\in (\MM^0_\mu)^{\wedge n}:\abs{v}\in\Lp^p(\mu)\right\}\qquad\text{if }p\in[1,\infty].
\end{equation} 
In particular, then, with the obvious interpretation for 	$\mathrm{L}^0_\capa(T^{\wedge n}\XX)$,
\begin{equation}\notag
	(\mathrm{L}^0_\capa(T^{\wedge n}\XX))^0_\mu\cong
	\mathrm{L}_\mu^0(T^{\wedge n}\XX)
\end{equation} 
so that
\begin{equation}\notag
	(\mathrm{L}_\capa^0(T^{\wedge n}\XX))^p_\mu\cong\mathrm{L}_\mu^p(T^{\wedge n}\XX)\qquad\text{if }p\in[1,\infty],
\end{equation} 
where $\mathrm{L}_\mu^p(T^{\wedge n}\XX)\defeq\left\{v\in \mathrm{L}_\mu^0(T^{\wedge n}\XX):\abs{v}\in\Lp^p(\mu)\right\}$.\fr
\end{rem}

\medskip
 Now we consider $\TestV(\XX)^{\wedge n}$, that is a module over the ring $\Ss\cap\Lpi$ in the algebraic sense and, by Lemma \ref{normsobowedge} below, 
\begin{equation}\notag
	\frac{1}{1\vee|v|}v\in\TestV(\XX)^{\wedge n}\qquad\text{for every }v\in\TestV(\XX)^{\wedge n}.
\end{equation}

As a consequence of of Lemma \ref{densitytestbargentens}, we have the following result.
\begin{lem}\label{densitytestbargentenswedge}
	Let $(\XX,\dist,\mass)$ be an $\RCD(K,\infty)$ space and let $\mu$ be a finite Borel measure such that $\mu\ll\capa$. Then
	$\TestV(\XX)^{\wedge n}$
	is dense in $\mathrm{L}_\mu^p(T^{\wedge n}\XX)$ for every $p\in[1,\infty)$.
\end{lem}
The following result corresponds to  Lemma \ref{normsobo}. 
\begin{lem}\label{normsobowedge}
	Let $(\XX,\dist,\mass)$ be an $\RCD(K,\infty)$ space and let $v\in(\WHCSs\cap\Lpi)^{\wedge n}$. Then $\abs{v}\in\HSs\cap\Lpi$.
\end{lem} 
\begin{proof}
	The proof is very similar to the one of Lemma \ref{normsobo}, we simply sketch the key computation for the sake of completeness. Take $v\in(\WHCSs\cap\Lpi)^{\wedge n}$, say $v=\sum_{i=1}^m v_1^i\otimes\dots\otimes v_n^i$, where $\{v_j^i\}\in \HSs\cap\Lpi$. Take $H_v\in\RR$ and $h_v\in\Lpt$ as in the proof of Lemma \ref{normsobo}. We will prove that
	$$|\nabla \abs{v}^2|\le 2 m n H_v^{n-1} h_v \abs{v}\qquad \mass\text{-a.e.}$$
	and thus the proof will be concluded as in Lemma \ref{normsobo}. We take $Z\in\tanXzero$ and we compute (here $S_n$ denotes the symmetric group)
	\begin{align*}
		&\nabla (v\,\cdot\, v)\,\cdot\, Z=\sum_{i,j=1}^m \nabla\det ((v^i_k v^j_h)_{h,k})=\sum_{i,j=1}^m\sum_{\sigma\in S_n} \nabla \sum_{k=1}^n v^i_k\,\cdot\, v^j_{\sigma (k)}\\
		&\qquad=2\sum_{i,j=1}^m\sum_{\sigma\in S_n} \sum_{h=1}^n (\nabla_Z v^i_h)\,\cdot\, v^j_{\sigma (h)}  \sum_{\substack{k=1\\ k\ne h}}^n v_k^i\,\cdot\,v_{\sigma(k)}^j
		\\
		&\qquad= 2\sum_{i,j=1}^m \left( \sum_{i=1}^m\sum_{h=1}^n v^i_1\wedge\cdots\wedge v_{h-1}^i\wedge \nabla_Z v_h^i\wedge v_{h+1}^i\wedge\cdots\wedge v^i_n\right)\,\cdot\,\left(\sum_{j=1}^m v^j_1\wedge\cdots\wedge v_n^j\right)
	\end{align*}
	so that we have proved the claim.
\end{proof}

\medskip 

As before, we consider the  multilinear map defined by \eqref{map}
and we notice that  (the left hand side is well defined thanks to  Lemma \ref{normsobowedge} - but as there is a squared norm here, this is indeed trivial)  \begin{equation}\label{equalitywedge}
	\qcr\left(\abs{ \sum_{i=1}^m v^i_1\wedge\cdots\wedge v^i_n}^2\right)=\abs{	\sum_{i=1}^m\qcrbar(v^i_1)\wedge\cdots\wedge\qcrbar(v^i_n)}^2 \qquad\capa\text{-a.e.}
\end{equation} so that the map in \eqref{map}
induces a map $\qcrbar:(\WHCSs\cap\Lpi)^{\wedge n}\rightarrow\mathrm{L}_\capa^0(T^{\wedge n}\XX)$ that satisfies, thanks to \eqref{equalitywedge},
\begin{equation}\notag
	\qcr({\abs{v}})=\abs{\qcrbar(v)} \qquad\capa\text{-a.e.\ for every }v\in (\WHCSs\cap\Lpi)^{\wedge n}.
\end{equation}
As usual, we often omit to write the maps $\qcrbar$ and $\qcr$.

\section{Main Part}
\subsection{Decomposition of the tangent module}
The following theorem provides us with a dimensional decomposition of the $\capa$-tangent module, along with an orthonormal basis  made of \say{smooth} vector fields of the $\capa$-tangent module on every element of the induced partition. This will be the first step towards the construction of the most relevant objects of this note. Notice that one should not expect the relevant dimension to be unique: in a smooth manifold of dimension $n$ with boundary, the $\capa$-tangent module sees the boundary, thus it has dimension $n$ in the interior of the manifold and dimension $n-1$ at the boundary. It is unclear if the situation on $\RCD$ spaces can be more complicated than that.
\begin{thm}[{\cite[Theorem 3]{BGgeneral}}]\label{captangent}
	Let $(\XX,\dist,\mass)$ an $\RCD(K,N)$ space of essential dimension $n$. Then there exists a partition of $\XX$ made of countably many bounded Borel sets $\{A_k\}_{k\in\NN}$ such that for every $k$ there exist $n(k)$ with $0\le n(k)\le n$ and $\big\{v_1^k,\dots,v^k_{n(k)}\big\}\subseteq\TestVbar(\XX)$ with bounded support which is an orthonormal basis of $\tanXcap$ on $A_k$, in the sense that $$ v_i\,\cdot\,v_j=\delta_i^j\qquad\capa\text{-a.e.\ on }A_k$$ and for every $v\in\tanXcap$ there exist $g_1,\dots,g_{n(k)}\in\Lpc$ such that 
		$$v=\sum_{i=1}^{n(k)}g_i v_i^k\qquad\capa\text{-a.e.\ on }A_k,$$
	where, in particular,  $$ g_i=v\,\cdot\,v_i^k\qquad\capa\text{-a.e.\ on }A_k.$$

	Here we implicitly state that if $n(k)=0$ then for every $v\in\tanXcap$ we have $v=0\ \capa$-a.e.\ on $A_k$.
\end{thm}

\subsection{Hessian}\label{secthess}
\subsubsection{Convexity}
In the following definition we restrict ourselves to the case of $\RCD$ spaces. This restriction is clearly unnecessary for items $(1)$ and $(2)$, however, we preferred this formulation for the sake of simplicity, taken into account that all the results of this note are in the framework of $\RCD$ spaces.
\begin{defn}\label{variousconvex}
Let $(\XX,\dist,\mass)$ be an $\RCD(K,\infty)$ space and let $f:\XX\rightarrow(-\infty,+\infty]$. Let also $\kappa\in\RR$.  We say that 
\begin{enumerate}
\item $f$ is weakly $\kappa$ geodesically convex if for every $x_0,x_1\in\XX$, there exists a constant speed geodesic $\gamma:[0,1]\rightarrow\XX$ joining $x_0$ to $x_1$ satisfying 
\begin{equation}\label{convgeod}
	f(\gamma(t))\le (1-t)f(\gamma(0))+t f (\gamma(1))-\frac{\kappa}{2} t(1-t)\dist(x_0,x_1)^2\qquad\text{for every }t\in[0,1].
\end{equation}
\item $f$ is strongly $\kappa$ geodesically convex  if for every $x_0,x_1\in\XX$, for every constant speed geodesic $\gamma:[0,1]\rightarrow\XX$ joining $x_0$ to $x_1$, \eqref{convgeod} holds.
\end{enumerate}
If moreover $f\in \HSloc(\XX)$, we say that
\begin{enumerate}
	\setcounter{enumi}{2}
\item $\hess f\ge \kappa$ if for every $h,g\in\TestFbs(\XX)$ with $h\ge 0\ \mass$-a.e.\
\begin{equation}\label{monotest}
	\int_\XX - \dive(h\nabla g)\nabla f\,\cdot\,\nabla g-\frac{1}{2}
	h \nabla f \,\cdot\,\nabla(\abs{\nabla g}^2)\dd{\mass}\ge \kappa \int_\XX \abs{\nabla g}^2 h\dd{\mass}.
\end{equation}
\end{enumerate}
\end{defn}
\begin{rem}\label{remconv}
Notice that if moreover $f\in\HSs$ and the space is locally compact, then item $(3)$ of the definition above implies that \eqref{monotest} holds for every $h\in\Ss\cap\Lpi$ and $g\in\TestF(\XX)$. This follows from an approximation argument, taking into account also the existence of good cut-off functions as in \cite[Lemma 6.7]{AmbrosioMondinoSavare13-2} together with the algebra property of test functions.\fr
\end{rem}
Some  implications among the various items of the previous definition have already been extensively studied in the literature, see e.g.\  \cite{Ketterer_2015, KK18,LiSt,HanMonot,Erbar-Kuwada-Sturm13,SturmGF} for similar statements. Notice that $(2)\Rightarrow(1)$  is  trivially satisfied in geodesic spaces. The implication $(1)\Rightarrow(3)$, recalled in Proposition \ref{kett} below, is particularly important in the sequel, as it motivates Theorem \ref{hessrepr}, one of the main results of this note. For the proof of Proposition \ref{kett} we are going to follow the lines of the proof of \cite[Theorem 7.1]{Ketterer_2015}. As we are going to work with weaker regularity assumptions, we give the details anyway. Indeed, the fact that we do not assume $f\in D(\hess)$ forces us to proceed through a delicate approximation argument.  Finally, under additional regularity assumptions, e.g.\ $(\XX,\dist,\mass)$ is a locally compact $\RCD(K,\infty)$ space and $f\in\TestF(\XX)$, it turns out that items $(1)$, $(2)$ and $(3)$ are all equivalent (see \cite{Ketterer_2015} and \cite{SturmGF}). The equivalence of these notions of convexity is expected to hold even under weaker assumptions on $f$ (see, in this direction, Proposition \ref{kett}) but this investigation (in particular $(3)\Rightarrow(1)$) is beyond the scope of this note. 

\begin{rem}{\rm
The implication $(3)\Rightarrow(1)$ seems anything but trivial, if one does not assume that $f\in D(\hess)$. Indeed, one could hope to follow \cite{Ketterer_2015} and start by proving that $(\XX,\dist,e^{-f}\mass)$ is $\RCD(K+\kappa,\infty)$ whenever  $(\XX,\dist,\mass)$ is $\RCD(K,\infty)$ and $\hess f\ge\kappa$.  In this context, the natural way to verify the  $\RCD(K+\kappa,\infty)$  condition is via the Eulerian point of view, i.e.\ via the weak Bochner inequality. However, in order to so, we would want to exploit an approximation argument, to plug in the weak Bochner inequality for $(\XX,\dist,\mass)$  and the fact that $\hess f\ge\kappa$ and such approximation argument seems to require that the heat flow on $(\XX,\dist,e^{-f}\mass)$  maps regular enough functions to Lipschitz functions, and we were not able to prove this fact (that we remark is linked with the $\RCD$ condition).}\fr\end{rem}
\begin{prop}\label{kett}
Let $(\XX,\dist,\mass)$ be an $\RCD(K,\infty)$ space and let $f \in\HSsloc$ be a continuous and weakly $\kappa$ geodesically convex function, for some $\kappa\in\RR$.
Assume moreover that $f$ is bounded from below and locally bounded from above, in the sense that $f$ is bounded from above on every bounded subset of $\XX$. Then $\hess f\ge \kappa$.
\end{prop}
\begin{proof}
	As remarked above, we  follow  the proof of \cite{Ketterer_2015}.
Define $\tilde{\mass}\defeq e^{-f}\mass$. When we want to stress that an object is relative to the space $(\XX,\dist,\tilde{\mass})$, we add the symbol $\tilde{\ }$. 
\medskip
\\\textbf{Step 1.} We show that $(\XX,\dist,\tilde{\mass})$ is an $\RCD(K+\kappa,\infty)$ space, following \cite[Proposition 6.19]{Ambrosio_2014}, which builds upon \cite[Proposition 4.14]{Sturm06I} and \cite[Lemma 4.11]{AmbrosioGigliSavare11}. 

To show $(K+\kappa)$-convexity of $\Ent_{\tilde{\mass}}$, first notice that the continuity of $f$ and classical measurable selection arguments (e.g.\ \cite[
6.9.13]{Bogachev072}) grant that there exists a $\mass\otimes\mass$-measurable map $\Gamma:\XX\times\XX\rightarrow\Geod(\XX)$ such that for $\mass\otimes\mass$-a.e.\ $(x_0,x_1)\in\XX\times\XX$, $\Gamma(x_0,x_1)$ is a geodesic joining $x_0$ to $x_1$ satisfying \eqref{convgeod}.
Also (\cite[Theorem 3.2]{DaneriSavare08}), the $\RCD(K,\infty)$ assumption on  $(\XX,\dist,\mass)$ implies that $\Ent_\mass$ is $K$-convex along every constant speed geodesic  $\{\mu_t\}_t\subseteq\Ptwo(\XX)$. Then, given $\mu,\nu\in\Ptwo(\XX)\cap D(\Ent_{\tilde{\mass}})$ we can argue as in \cite[Proposition 4.14]{Sturm06I}, verifying $(K+\kappa)$-convexity of  $\Ent_{\tilde{\mass}}$ along the geodesic given by $\{(e_t\circ \Gamma)_*(\mu\otimes\nu)\}_t$.
\medskip
\\\textbf{Step 2.}  Notice that, as $f$ is locally bounded, \cite[Lemma 4.11]{AmbrosioGigliSavare11} implies that $\phi\in\HSsloc$ if and only if $\phi\in\tilde{\mathrm{H}}_{\mathrm{loc}}^{1,2}(\XX)$ and, if this is the case
$$\abs{\nabla \phi}=|\tilde{\nabla} \phi|.$$
Also, polarizing, we obtain that the $\,\cdot\,$ product between gradients is independent of the space, so that we will drop the $\tilde{\ }$ on gradients.
Moreover, if $\phi\in\HSs$, then $\phi\in\tilde{\mathrm{H}}^{1,2}(\XX)$ and, if $\phi\in \TestFbs(\XX)$, then $\phi\in D(\tilde{\Delta})$ and 
\begin{equation}\label{laplacians}
\tilde{\Delta} \phi=\Delta \phi-\nabla f\,\cdot\,\nabla \phi. 
\end{equation}

By the equivalence result in \cite{AmbrosioGigliSavare12} (see also \cite{AmbrosioMondinoSavare13,Erbar-Kuwada-Sturm13}), we know that if $k,g\in D(\tilde{\Delta})$  with $\tilde{\Delta}k\in \Lpi$, $\tilde{\Delta}g\in \tilde{\mathrm{H}}^{1,2}(\XX)$ and 
$k\in\Lpi$, $k\ge 0$,
\begin{equation}\notag
	\begin{split}
(K+\kappa)\int_\XX\abs{{\nabla }g}^2 k \dd{\tilde{\mass}}
&\le \frac{1}{2}\int_\XX \abs{\nabla g}^2 \tilde{\Delta }k\dd{\tilde{\mass}}-\int_\XX (\nabla g\,\cdot\,\nabla \tilde{\Delta} g) k\dd{\tilde{\mass}}\\
&=\frac{1}{2}\int_\XX \abs{\nabla g}^2 \tilde{\Delta }k\dd{\tilde{\mass}}+\int_\XX\tilde{\dive}( k \nabla g)\tilde{\Delta} g \dd{\tilde{\mass}} \\
&=\frac{1}{2}\int_\XX \abs{\nabla g}^2 \tilde{\Delta }k\dd{\tilde{\mass}}+\int_\XX \nabla k\,\cdot\,\nabla g\,\tilde{\Delta} g \dd{\tilde{\mass}}+\int_\XX  k(\tilde{\Delta} g)^2\dd{\tilde{\mass}}. 
\end{split}
\end{equation}
By an approximation argument based on the mollified heat flow (for the space $(\XX,\dist,\tilde\mass)$) on $g$, we can use what we just proved to show that if $g\in\TestFbs(\XX)$ and $k$ is as above,
\begin{equation}\notag
	\begin{split}
		(K+\kappa)\int_\XX\abs{{\nabla }g}^2 k \dd{\tilde{\mass}}
		&\le\frac{1}{2}\int_\XX \abs{\nabla g}^2 \tilde{\Delta }k\dd{\tilde{\mass}}+\int_\XX \nabla k\,\cdot\,\nabla g\tilde{\Delta} g \dd{\tilde{\mass}}+\int_\XX  k(\tilde{\Delta} g)^2\dd{\tilde{\mass}}
		\\&= -\frac{1}{2}\int_\XX \nabla \abs{\nabla g}^2\,\cdot\, \nabla k \dd{\tilde{\mass}}+\int_\XX \nabla k\,\cdot\,\nabla g\tilde{\Delta} g \dd{\tilde{\mass}}+\int_\XX  k(\tilde{\Delta} g)^2\dd{\tilde{\mass}}.
	\end{split}
\end{equation}

Then, with an approximation argument based on the mollified heat flow on $k$, we have that if $g\in\TestFbs(\XX)$ and $k\in\HSs\cap\Lpi$, it holds that
\begin{equation}\notag
\begin{split}
		(K+\kappa)\int_\XX\abs{{\nabla }g}^2 k \dd{\tilde{\mass}}\le -\frac{1}{2}\int_\XX \nabla \abs{\nabla g}^2\,\cdot\, \nabla k \dd{\tilde{\mass}}+\int_\XX \nabla k\,\cdot\,\nabla g\tilde{\Delta} g \dd{\tilde{\mass}}+\int_\XX  k(\tilde{\Delta} g)^2\dd{\tilde{\mass}}.
\end{split}
\end{equation}
We choose then $k=h e^{f}$ to obtain ($h\in\TestFbs(\XX)$), recalling \eqref{laplacians},
\begin{equation}\notag
	\begin{split}
		(K+\kappa)\int_\XX\abs{{\nabla }g}^2 h\dd{\mass}
		&\le -\frac{1}{2}\int_\XX \nabla \abs{\nabla g}^2\,\cdot\, \nabla h \dd{\mass}  -\frac{1}{2}\int_\XX ( \nabla\abs{\nabla g}^2\,\cdot\, \nabla f) h \dd{\mass}\\&\qquad+\int_\XX \nabla h\,\cdot\,\nabla g (\Delta g -\nabla f\,\cdot\,\nabla g)\dd{\mass}\\&\qquad+\int_\XX \nabla f\,\cdot\,\nabla g (\Delta g -\nabla f\,\cdot\,\nabla g)h \dd{\mass}
		\\&\qquad+ \int_\XX h (\Delta g -\nabla f\,\cdot\,\nabla g)^2\dd{\mass}
		\\&=-\frac{1}{2}\int_\XX \nabla \abs{\nabla g}^2\,\cdot\, \nabla h \dd{\mass}  -\frac{1}{2}\int_\XX  (\nabla\abs{\nabla g}^2\,\cdot\, \nabla f )h \dd{\mass}\\
		 &\qquad+\int_\XX \dive(h \nabla g)\Delta g\dd{\mass}
-\int_\XX \dive(h\nabla g)\nabla g\,\cdot\,\nabla f\dd{\mass}.
	\end{split}
\end{equation}
\medskip
\\\textbf{Step 3.} Let now $\alpha>0$. We repeat the same computation of \textbf{Step 2} but starting from the $\RCD(\alpha^2 K,\infty)$ space $(\XX,\alpha^{-1}\dist,\mass)$ and the $ \kappa$ convex function $ \alpha^{-2}f$ (of course, with respect to $\alpha^{-1}\dist$) to obtain (all the differential operators are with respect to $(\XX,\dist,\mass)$)
\begin{equation}\notag
	\begin{split}
		(\alpha^2 K+\kappa)\int_\XX \alpha^2\abs{{\nabla }g}^2 h\dd{\mass}
		&\le -\alpha^4\frac{1}{2}\int_\XX \nabla \abs{\nabla g}^2\,\cdot\, \nabla h \dd{\mass}-\alpha^2\frac{1}{2}\int_\XX \nabla \abs{\nabla g}^2\,\cdot\, \nabla  f h \dd{\mass} \\&	\qquad +\alpha^4\int_\XX \dive(h \nabla g)\Delta g\dd{\mass}
	 -\alpha^2\int_\XX \dive(h\nabla g)\nabla g\,\cdot\,\nabla f\dd{\mass}.
	\end{split}
\end{equation}
Dividing this inequality by $\alpha^2$ and letting $\alpha\searrow 0$ yields the claim.
\end{proof}

\subsubsection{Measure valued Hessian}
In this section we state and prove the first main result of this note, namely Theorem \ref{hessrepr}. More precisely, we show that convex functions have, in a certain sense, a measure valued Hessian. In the Euclidean space, this is an immediate consequence of Riesz's Theorem for positive functionals and it implies that gradients of convex functions are vector fields of bounded variation. Hence we have that Hessian measures are absolutely continuous with respect to $\capa$, and  this is the case even on $\RCD$ spaces. This absolute continuity allows us to build the measure valued Hessian on $\RCD$ spaces as product of a $\capa$-tensor field and a $\sigma$-finite measure  that is absolutely continuous with respect to $\capa$.  We remark that, as the decomposition of the $\capa$-tangent module given by Theorem \ref{captangent} induces a decomposition of the space in Borel sets (not open ones), we are not able to prove that the total variation of the Hessian measure is a Radon measure. 

\medskip

Before dealing with the main theorem of this section, we define when a $\HSs$ function has a measure valued Hessian (cf.\ \cite[Definition 3.3.1]{Gigli14}) and study a couple of basic calculus properties of this newly defined notion. 
\begin{defn}\label{disthessdef}
Let $(\XX,\dist,\mass)$ be an $\RCD(K,\infty)$ space and $f\in\HSsloc$. We write  $f\in D(\Hhess)$, if there exists a $\sigma$-finite measure $\abs{\Hhess f}$ that satisfies $\abs{\Hhess f}\ll\capa$ and a symmetric tensor field $\nu_f\in\tanXcaptwo$ with $\abs{\nu_f}=1\ \abs{\Hhess f}$-a.e.\ such that 
\begin{equation}\notag
\Hhess f=\nu_f\abs{\Hhess f},
\end{equation} in the sense that for every $X,Y\in\WHHSs\cap\tanXinf$, it holds that $ X\otimes Y\,\cdot\,\nu_f\in\Lploc^1(\abs{\Hhess f})$ and, if $h\in\HSs\cap\Lpi$ has bounded support, 
\begin{equation}\label{bigclaim}
	\int_\XX hX\otimes Y\,\cdot\,\nu_f\dd{\abs{\Hhess f}}=-\int_\XX \nabla f\,\cdot\, X\dive(h Y)+h\nabla_{Y} X\,\cdot\,\nabla f\dd{\mass}.
\end{equation}
\end{defn}
\begin{prop}\label{uniquenesshess}
Let $(\XX,\dist,\mass)$ be an $\RCD(K,\infty)$ space and let $f\in D(\Hhess)$. Then the decomposition of $\Hhess f$ is unique, in the sense that, adopting the same notation of Definition \ref{disthessdef}, the measure $\abs{\Hhess f}$ is unique and the tensor field $\nu_f$ is unique, up to $\abs{\Hhess f}$-a.e.\ equality.
\end{prop}
\begin{proof} Assume that $(\nu,\mu)$ and $(\nu',\mu')$ are two couples having the same properties of the couple $(\nu_f,\abs{\Hhess f})$ as in Definition \ref{disthessdef}. We show that $\mu=\mu'$ and that $\nu=\nu'$ $\mu$-a.e. By dominated convergence, we have that 
$$\int_K X\otimes Y\,\cdot\,\nu\dd{\mu}=\int_K X\otimes Y\,\cdot\,\nu'\dd{\mu'}$$
for every $K\subseteq\XX$ compact and $X,Y\in\WHHSs\cap\tanXinf$. 
This shows, by Lemma \ref{densitytestbargentens}, that $\mu\mres K=\mu'\mres K$ and $\nu=\nu'\ \mu$-a.e.\ on $K$ for any compact set $K$ satisfying $\mu(K),\mu'(K)<\infty$, and this is clearly enough to conclude.
\end{proof}

Using the by now classical calculus tools on $\RCD$ spaces, the following proposition easily follows, starting from \eqref{bigclaim}.
\begin{prop}
Let $(\XX,\dist,\mass)$ be an $\RCD(K,\infty)$ space. Then
\begin{enumerate}[label=\roman*)]
	\item $D(\hess)\subseteq D(\Hhess)$. Namely, if $f\in D(\hess)$, then $\Hhess f=\frac{\hess f}{\abs{\hess f}}(\abs{\hess f}{\mass})$;
\item $D(\Hhess)$ is a vector space. Namely, if $f,g\in D(\Hhess)$, say $\Hhess f=\nu_f\abs{\Hhess f}$, $\Hhess g=\nu_g\abs{\Hhess g}$, then $f +g\in D(\Hhess)$ and, setting 
\begin{align*}
	\mu&\defeq \abs{\Hhess f}+\abs{\Hhess g},\\
	\nu&\defeq\nu_f\dv{\abs{\Hhess f}}{\mu}+\nu_g\dv{\abs{\Hhess g}}{\mu},
\end{align*}
it holds that $\Hhess (f +g)=\frac{\nu}{\abs{\nu}}(\abs{\nu}\mu)$;
\item $D(\Hhess)\cap\Lpiloc$ is an closed under multiplication. Namely, if $f,g\in D(\Hhess)\cap\Lpiloc$, say $\Hhess f=\nu_f\abs{\Hhess f}$, $\Hhess g=\nu_g\abs{\Hhess g}$, then $f g\in D(\Hhess)$ and, setting 
\begin{align*}
\mu&\defeq \abs{\Hhess f}+\abs{\Hhess g}+\mass,\\
 \nu&\defeq g\nu_f\dv{\abs{\Hhess f}}{\mu}+f\nu_g\dv{\abs{\Hhess g}}{\mu}+\big(\nabla f\otimes\nabla g
+\nabla g\otimes\nabla f\big)\dv{\mass}{\mu},
\end{align*}
it holds that $\Hhess (f g)=\frac{\nu}{\abs{\nu}}(\abs{\nu}\mu);$
\item $D(\Hhess)\cap\Lpiloc$ is an closed under post-composition with $C^2$ functions. Namely, if $f\in D(\Hhess)$, say $\Hhess f=\nu_f\abs{\Hhess f}$ and $\varphi\in C^2(\RR)$, then $\varphi\circ f\in D(\Hhess)$ and, setting 
\begin{align*}
	\mu&\defeq \abs{\Hhess f}+\mass,\\
	\nu&\defeq\varphi'\circ f\nu_f\dv{\abs{\Hhess f}}{\mu}+\varphi''\circ f\nabla f\otimes\nabla  f\dv{\mass}{\mu},
\end{align*}
it holds that $\Hhess (\varphi\circ f)=\frac{\nu}{\abs{\nu}}(\abs{\nu}\mu).$
\end{enumerate}
\end{prop}

\medskip

Now we show that convex functions (recall Definition \ref{variousconvex}) have measure valued Hessian. As a notation, here and below, $a^-\defeq -a\vee 0$ for every $a\in\RR$.
\begin{thm}\label{hessrepr}
Let $(\XX,\dist,\mass)$ be an $\RCD(K,N)$ space and $f\in\HSsloc$ satisfying $\hess f\ge \kappa$, for some $\kappa\in\RR$.
Then, $f\in D(\Hhess)$, say  $\Hhess f=\nu_f\abs{\Hhess f}$.
Moreover, we have that
\begin{equation}\notag
\nu_f\abs{\Hhess f}\ge \kappa \mathrm{g}\mass,
\end{equation}
in the sense that for every $v\in\tanXcap$, it holds, as measures,
 \begin{equation}\label{ineqmeas}
   v\otimes v\,\cdot\,\nu_f \abs{\Hhess f}\ge \kappa\abs{v}^2{\mass}.
 \end{equation}
 Finally, if  also $f\in\HSs$, then for every $X,Y\in\WHHSs\cap\tanXinf$, we have that $ X\otimes Y\,\cdot\,\nu_f\in\Lp^1(\abs{\Hhess f})$,  in particular, \eqref{bigclaim} holds for every $h\in\Ss\cap\Lpi$ and $X,Y\in\WHHSs\cap\tanXinf$. More precisely, we have the explicit bound, for every $X\in\WHHSs\cap\tanXinf$
	\begin{equation}
		\label{eq:boundmass}
		\int_\XX |X\otimes X\,\cdot\,\nu_f|\,\dd\abs{\Hhess f}\leq \int_\XX -\dive\,X\nabla f\cdot X-\nabla f\cdot\nabla(\tfrac12|X|^2)+2\kappa^-|X|^2\,\dd\mm.
	\end{equation}
\end{thm}
\begin{proof}
We divide the proof in several steps.
\medskip
\\\textbf{Step 1.} We define for  $h\in\Ss\cap\Lpi$ with bounded support and $X,Y\in\WHHSs\cap\tanXinf$,
$$
\mathcal{G}_f(h,X,Y)\defeq\int_\XX - \frac{1}{2}\dive(h X)\nabla f\,\cdot\, Y-\frac{1}{2}\dive(h Y)\nabla f\,\cdot\, X-\frac{1}{2}h \nabla f \,\cdot\,\nabla( X\cdot Y)\dd{\mass}
$$
and we write for simplicity $$ \mathcal{G}_f(h,X)\defeq  \GG_f(h,X,X).$$ We shall frequently use the fact that for given $f,h,Y$ as above
\begin{equation}
\label{eq:contG}
\begin{split}
&\text{the map $X\mapsto  \mathcal{G}_f(h,X,Y)\in\RR$ is continuous w.r.t.\ the $\WHHSs$-norm}\\
&\text{on sets of vector fields with uniformly bounded $L^\infty$-norm,}
\end{split}
\end{equation} and similarly for $Y$.

Notice that $\mathcal{G}_f(h,X,Y)$ equals the right hand side of \eqref{bigclaim} for  $X,Y\in\WHHSs\cap\tanXinf$ and $h\in\Ss\cap\Lpi$, as a consequence of a simple approximation argument on $f$.
Indeed, as $h$ has bounded support, a locality argument shows that it is not restrictive to assume also $f\in\HSs$, then we can approximate $f$ in the $\HSs$ topology with functions in $\TestF(\XX)$ and see the equality of the two quantities. This shows also that, as $\hess f\ge \kappa$,
\begin{equation}\label{monoton}
	\mathcal{G}_f(h,\nabla g)\ge \kappa\int_\XX h\abs{\nabla g}^2\dd{\mass}
\end{equation}  if $h\in\TestF(\XX)$ is non negative and has bounded support and $g\in\TestF(\XX)$ (by approximation, \eqref{monoton} continues to hold if  $h\in\Ss\cap\Lpi$ is non negative and  has bounded support). 
This argument also shows that  
\begin{equation}\label{bilinapprox}
	\sum_{i,j=1}^m \GG_f( h f_i g_j,X_i,Y_j)=\sum_{i,j=1}^m \GG_f( h  ,f_i X_i,g_j Y_j)
\end{equation}
for every $h\in\Ss\cap\Lpi$ with bounded support, $\{f_i\}_i,\{g_i\}_i\subseteq\HSs\cap\Lpi$ and $\{X_i\}_i,\{Y_i\}_i\subseteq\WHHSs\cap\tanXinf$ (we are implicitly using Lemma \ref{calculushodge}).
Clearly, $\GG_f(h,\,\cdot\,,\,\cdot\,)$ is symmetric and $\RR$-bilinear for every $h\in\Ss\cap\Lpi$.
\medskip
\\\textbf{Step 2.} By \eqref{monoton}, the Riesz–-Daniell–-Stone Theorem yields that for every $g\in\TestF(\XX)$ there exists a unique Radon measure $\mu_{\nabla g}$ such that
\begin{equation}\label{defmeas}
\mathcal{G}_f(h,\nabla g)=\int_\XX h\dd{\mu_{\nabla g}}\qquad \text{for every }h\in\LIPbs (\XX).
\end{equation} 
Recalling \eqref{monoton}, we have that 
\begin{equation}\label{monoton1}
\mu_{\nabla g}\ge \kappa\abs{\nabla g}^2\dd{\mass}.
\end{equation}
We show now that $\mu_{\nabla g}\ll\capa$. Being $\mu_{\nabla g}$ a Radon measure, it is enough to show that if $K$ is a compact set such that $\capa(K)=0$, then $\mu_{\nabla g}(K)=0$. As $K$ is compact, we can find a sequence $\{u_n\}_n\subseteq\HSs\cap\LIPb(\XX)$ with $u_n= 1$ on a neighbourhood $K$, $0\le u_n\le 1$ and $\Vert u_n\Vert_{\HSs}\rightarrow 0$ (see e.g.\ \cite[Lemma 5.4]{Bjorn-Bjorn11} or \cite[Lemma 2.3]{BGBV}). 
Also, it is easy to see that we can assume with no loss of generality also that $\{u_n\}_n\subseteq\LIPbs(\XX)$ have uniformly bounded support.
Therefore, using \eqref{monoton1}, dominated convergence and \eqref{defmeas}, 
$$0\leq \mu_{\nabla g}(K)-\kappa\int_K \abs{\nabla g}^2\dd{\mass}\le \int_\XX u_n\dd{\mu_{\nabla g}}-\kappa\int_\XX u_n \abs{\nabla g}^2\dd{\mass}\stackrel{\eqref{eq:contG}}\rightarrow 0,$$
having used also that $\mm(K)$=0 in the last step. In particular, $ \mu_{\nabla g}(K)= \mu_{\nabla g}(K)-\kappa\int_K \abs{\nabla g}^2\dd{\mass}$ and thus the above proves that $\mu_{\nabla g}\ll\capa$, as claimed. 

As $\mu_{\nabla g}\ll\capa$, using dominated convergence, we can show that $$ \mathcal{G}_f(h,\nabla g)=\int_\XX h\dd{\mu_{\nabla g}}\qquad \text{for every }h\in\Ss\cap\Lpi\text{ with bounded support}$$ 
where we implicitly take the quasi continuous representative of $h$.
We define by polarization the signed Radon measure, absolutely continuous with respect to $\capa$,
$$ \mu_{\nabla g_1,\nabla g_2}\defeq \frac{1}{4}\left(\mu_{\nabla(g_1+g_2)}-\mu_{\nabla(g_1-g_2)}\right)\qquad\text{if }g_1,g_2\in\TestF(\XX)$$
so that, by the properties of $(X,Y)\mapsto\GG_f(h,X,Y)$, $$ \mathcal{G}_f(h,\nabla g_1,\nabla g_2)=\int_\XX h\dd{\mu_{\nabla g_1,\nabla g_2}}\qquad \text{for every }h\in\HSs\cap\Lpi\text{ with bounded support}.$$
Notice that the map $(g_1,g_2)\mapsto \mu_{\nabla g_1,\nabla g_2}$ is symmetric  and $\RR$-bilinear by its very definition.
\medskip
\\\textbf{Step 3.} We show that for every $X\in\WHHSs\cap\tanXinf$ and $h\in\HSs\cap\Lpi$ non negative and with bounded support, 
\begin{equation}\label{tmp11}
\mathcal G_f(h,X,X)\ge \kappa \int_\XX h\abs{X}^2\dd{\mass}.
\end{equation}
By dominated convergence and \cite[Lemma 3.2]{BGBV}, we see that it is enough to assume $X\in\TestV(\XX)$, say 
	\begin{equation}\notag
	X=\sum_{i=1}^m f_i\nabla g_i\qquad\text{with }\{f_i\}_i\subseteq\Ss\cap\Lpi\text{ and }\{g_i\}_i\subseteq\TestF(\XX).
\end{equation}

By  the properties of the map $(X,Y)\mapsto\GG_f(h,X,Y)$ (in particular, recall \eqref{bilinapprox}) we see that \eqref{tmp11} will follow from 
$$\sum_{i,j=1}^m \GG_f( h f_i f_j,\nabla g_i,\nabla g_j)\ge \kappa\int_\XX h\abs{X}^2\dd{\mass}.$$
It suffices then to show then that, as measures,
$$ \sum_{i,j=1}^m  f_i f_j\mu_{\nabla g_i,\nabla g_j}\ge \kappa\sum_{i,j=1}^m f_i f_j\nabla g_i\,\cdot\,\nabla g_j{\mass}.$$
By dominated convergence and localizing, we further reduce to the case in which $f_i=c_i\in\RR$ for every $i=1,\dots,m$, this is to say, as measures,
$$ \sum_{i,j=1}^m  c_i c_j\mu_{\nabla g_i,\nabla g_j}\ge \kappa \sum_{i,j=1}^m c_i c_j\nabla g_i\,\cdot\,\nabla g_j{\mass},$$
that follows by \eqref{monoton1} with $\sum_{i=1}^m c_i g_i$ in place of $g$.
\medskip
\\\textbf{Step 4.} Building upon \eqref{tmp11} and arguing as in \textbf{Step 2}, we can define the Radon measure $\mu_{X,Y}$ whenever $X,Y\in\WHHSs\cap\tanXinf$, such that 
$$
\int_\XX h\dd{\mu_{X,Y}}=\mathcal G_f(h,X,Y)\qquad\text{for every }h\in\LIPbs(\XX).
$$
More precisely, we first define $\mu_{X,X}$ for $X\in\WHHSs\cap\tanXinf$ and then we define $\mu_{X,Y}$ for $X,Y\in\WHHSs\cap\tanXinf$ by polarization, taking into account that  $\GG_f(h,X,Y)$ is symmetric in $X$ and $Y$.  
 
By the properties of $(X,Y)\mapsto\GG_f(h,X,Y)$ (in particular, recall \eqref{bilinapprox}) it follows that, if for some $m$ and $l=1,2$, $$X_l=\sum_{i=1}^m f_i^l\nabla g_i^l\qquad\text{with }\{f_i^l\}_i\subseteq\Ss\cap\Lpi\text{ and }\{g_i^l\}_i\subseteq\TestF(\XX),$$
then we have that $$	\mu_{X_1,X_2}=\sum_{i,j=1}^m f_i^1 f_j^2\mu_{\nabla g_i^1,\nabla g_j^2}.$$
In particular, this definition is coherent with the one given in \textbf{Step 2}.  

Recall that, by \eqref{tmp11}, if  $X\in \WHHSs\cap\tanXinf$,
\begin{equation}\label{positv}
\mu_{X,X}\ge \kappa\abs{X}^2\mass.
\end{equation}
Also, as in \textbf{Step 2}, we show that $$\mu_{X,Y}\ll\capa,$$ whenever $X,Y\in\WHHSs\cap\tanXinf$.
\medskip
\\\textbf{Step 5.} We use now Theorem \ref{captangent} to take a partition $\{A_k\}_{k\in\NN}$ and, for any $k$, an orthonormal basis of $\tanXcap$ on $A_k$  $v^k_1,\dots, v^k_{n(k)}$. Fix for the moment $k$ and define, for $i,j=1,\dots,n(k)$, $$\mu_{i,j}^k\defeq\mu_{v^k_i,v^k_j}\mres A_k\qquad\text{and}\qquad\mu^k\defeq \sum_{i,j=1}^{n(k)} |{\mu_{i,j}^k}|.$$ Notice that this is a good definition as $\Prbar(v^k_i)\in \WHHSs$ for every $i=1,\dots,n(k)$ and that the measures above are finite signed measures absolutely continuous with respect to $\capa$. 

We define $$\tilde{\nu}^k_f\defeq \sum_{i,j=1}^{n(k)} v_i^k\otimes v_j^k\dv{\mu_{i,j}^k}{\mu^k}$$
then $$\nu^k_f\defeq\frac{1}{|{\tilde{\nu}^k_f}|} \tilde{\nu}^k_f\qquad\text{and}\qquad|{\Hhess^k f}|\defeq {|{\tilde{\nu}^k_f}|}\mu^k.$$
Finally, $$\nu_f\defeq \sum_k\nu_f^k\qquad\text{and}\qquad|{\Hhess f}|\defeq\sum_k|{\Hhess^k f}|.$$
Clearly, $\abs{\Hhess f}$ is a $\sigma$-finite measure, $\abs{\Hhess}\ll\capa$, $\abs{\nu_f}=1\ \abs{\Hhess f}$-a.e.\ and $\nu_f$ is symmetric.
\medskip
\\\textbf{Step 6.} Let $X,Y\in\WHHSs\cap\tanXinf$ and $h\in\HSs\cap\Lpi$ with bounded support. We verify that $ X\otimes Y\,\cdot\,\nu_f\in\Lploc^1(\abs{\Hhess f})$ and that \eqref{bigclaim} holds. By polarization, there is no loss of generality in assuming that $X=Y$ and we can also assume, by linearity, that $h$ is non negative. Notice indeed that the right hand side of \eqref{bigclaim} is equal to $\GG_f(h,X,Y)$ (recall \textbf{Step 1}) and hence is symmetric in $X,Y$. 

Consider the Borel partition $\{A_k\}_k$ as in Theorem \ref{captangent}. 
Assume for the moment that for every $k$, for every $h\in\LIPbs(\XX)$,
\begin{equation}\label{ambclaim}
	\int_{A_k} h\dd{\mu_{X,X}}=\int_{A_k} h X\otimes X\,\cdot\,\nu_f\dd{\abs{\Hhess f}}
\end{equation}
(notice that the restriction of ${\abs{\Hhess f}}$ to $A_k$ is the finite measure ${\abs{\Hhess^k f}}$, so the right hand side is well defined).
Then, it holds that 
\begin{equation}\label{ambclaimabs}
	\abs{X\otimes X\,\cdot\,\nu_f}\abs{\Hhess f}=\abs{\mu_{X,X}}
\end{equation}
that yields local integrability.
We can then compute, by dominated convergence and \eqref{ambclaim} (recall that $h$ has bounded support), 
\begin{equation}\notag
\begin{split}
\GG_f(h,X,X)&=\int_\XX h\dd{\mu_{X,X}}=\sum_k \int_{A_k} h\dd{\mu_{X,X}}\\&=\sum_k \int_{A_k} h X\otimes X\,\cdot\,\nu_f\dd{\abs{\Hhess f}}=\int_\XX h X\otimes X\,\cdot\,\nu_f\dd{\abs{\Hhess f}},
\end{split}
\end{equation}
that is \eqref{bigclaim}.

We show then \eqref{ambclaim}. Fix $k$ and recall the notation of \textbf{Step 6}. Notice that, considering the left hand side of \eqref{ambclaim}, we have, by the very definition of $\nu_f$ and $\abs{\Hhess f}$, on $A_k$
\begin{equation}\notag
\begin{split}
 X\otimes X\,\cdot\,\nu_f{\abs{\Hhess f}}=\sum_{i,j=1}^{n(k)} X\otimes X\,\cdot\, v_i^k\otimes v_j^k{\mu_{i,j}^k}=\sum_{i,j=1}^{n(k)} X\,\cdot\, v_i^k X\,\cdot\, v_j^k {\mu_{v_i^k,v_j^k}^k}=\mu_{\tilde{X},\tilde{X}},
\end{split}
\end{equation}
where $\tilde{X}\defeq \sum_{i=1}^{n(k)}( X\,\cdot\, v_i^k) v_i^k\in\TestV(\XX)$ satisfies then $\tilde{X}=X\ \capa$-a.e.\ on $A_k$.

As we have reduced ourselves to check that $\mu_{X,X}\mres A_k=\mu_{\tilde{X},\tilde{X}}\mres A_k$, taking into account also the bilinearity of the map $$(\WHHSs\cap\tanXinf)^2\ni (X,Y)\mapsto \mu_{X,Y},$$ we see that it is enough to show that
\begin{equation}\notag
\text{ for every $X,Y\in\WHHSs\cap\tanXinf$ we have}\qquad \mu_{X,Y}\mres \{\abs{X}=0\}=0.
\end{equation}

Let $\{\varphi_n\}_n\subseteq\HSs$  be as in Lemma \ref{Amb} for the vector field $X\in\WHHSs\cap\tanXinf$.
We compute, if $h\in\LIPbs(\XX)$, by \eqref{bilinapprox},
\begin{equation}\notag
	\int_\XX h\varphi_n \dd{\mu_{X,Y}}=\GG_f(h\varphi_n,X,Y)=\GG_f(h,\varphi_nX,Y).
\end{equation}
By the very definition, $\varphi_n(x)\searrow \chi_{\{\abs{X}=0\}}\ \capa$-a.e.\ so that, by dominated convergence, 
$$ \int_\XX h\varphi_n \dd{\mu_{X,Y}}\rightarrow \int_\XX h\chi_{\{\abs{X}=0\}}\dd{\mu_{X,Y}}.$$
On the other hand, as $\varphi_n X\rightarrow 0$ in the $\WSHs$ topology (Lemma \ref{Amb}), we have $$\GG_f(h,\varphi_nX,Y)\rightarrow 0$$ by \eqref{eq:contG}. Therefore, for every $h\in\LIPbs(\XX)$,  $$\int_\XX h\chi_{\{\abs{X}=0\}}\dd{\mu_{X,Y}}=0,$$ which means $\mu_{X,Y}\mres\{\abs{X}=0\} =0$.
%
%
\medskip
\\\textbf{Step 7.} We prove \eqref{ineqmeas}. By a locality argument, we reduce ourselves to the case $v\in\tanXcapinf$. By density and dominated convergence, it is enough to show \eqref{ineqmeas} for $v\in\WHHSs\cap\tanXinf$. By \eqref{positv} and \eqref{ambclaim} , if $v\in\WHHSs\cap\tanXinf$, $$v\otimes v \,\cdot\,\nu_f\abs{\Hhess f}=\mu_{v,v}\ge \kappa\abs{v}^2\mass,$$ that proves the claim.
\medskip
\\\textbf{Step 8.} We prove the last claim. Again, we assume with no loss of generality that $X=Y$. It is enough to show that if moreover $f\in\HSs$, then  $ X\otimes X\,\cdot\,\nu_f\in\Lp^1(\abs{\Hhess f})$, then the rest will follow from dominated convergence. 
The integrability follows from \eqref{ambclaimabs} if we show that $\mu_{X,X}$ is a finite signed measure. Inequality \eqref{positv} implies that the measure $\mu_{X,X}-\kappa \abs{X}^2\mass$ is non negative, but now, using an immediate approximation argument and monotone convergence together with \eqref{bigclaim}, we see that it is also finite. As $\abs{X}^2\mass$ is a finite measure, we see that $\mu_{X,X}$ is a finite signed measure and that \eqref{eq:boundmass} holds.
\end{proof}
\subsection{Ricci tensor}\label{sectionric}
As done for the Hessian, we give a fine meaning the the Ricci tensor defined in \cite{Gigli14}. Namely, we represent the Ricci tensor as a product of a $\capa$-tensor field and a $\sigma$-finite measure that is absolutely continuous with respect to $\capa$. As in the proof of Theorem \ref{hessrepr}, we are going to use a version of Riesz representation Theorem for positive functional, this time leveraging on the bound from below for the Ricci tensor ensured, in a synthetic way, by the  definition of the $\RCD$ condition. 

\medskip

 We recall  now the distributional definition of the objects that we are going to need and, to this aim, we recall  also that the definition of $\VV$ is in \eqref{oldtest}.
\begin{thm}[{\cite[Theorem 3.6.7]{Gigli14}}]\label{GigRic}
Let $(\XX,\dist,\mass)$ be an $\RCD(K,\infty)$ space. There exists a unique continuous map 
\begin{equation}\notag
\Ric:\WHHSs^2\rightarrow\Meas(\XX)
\end{equation}
such that for every 
$X,Y\in \VV$ it holds
\begin{equation}\label{rictest}
\Ric(X,Y)\defeq\DDelta \frac{X\,\cdot\, Y}{2}+\left( \frac{1}{2} X\,\cdot\,\DeltaH Y+ \frac{1}{2} Y\,\cdot\,\DeltaH X-\nabla X\,\cdot\,\nabla Y\right)\mass.
\end{equation}
Such map is bilinear, symmetric and satisfies
\begin{align}
\Ric(X,X)&\ge K\abs{X}^2\mass; \label{ricpos}\\
\int_\XX \dd{\Ric(X,Y)}&=\int_\XX \diff X\,\cdot\diff Y+\delta X\,\cdot\,\delta Y-\nabla X\,\cdot\,\nabla Y\dd{\mass};\notag\\
\Vert \Ric(X,Y)\Vert_{\TV} &\le2 \sqrt{\EE_\HG(X)+K^-\Vert X\Vert^2_{\tanX}}\sqrt{\EE_\HG(Y)+K^-\Vert Y\Vert^2_{\tanX}};\notag
\end{align}
for every $X,Y\in\WHHSs$.
\end{thm}

\begin{thm}\label{mainricc}
 Let $(\XX,\dist,\mass)$ be an $\RCD(K,N)$ space. 
 Then there exists a unique $\sigma$-finite measure $\abs{\Ric}$ that satisfies $\abs{\Ric}\ll\capa$ and a unique, up to  $\abs{\Ric}$-a.e.\ equality, symmetric tensor field $\omega\in\tanXcaptwo$ with $\abs{\omega}=1\ \abs{\Ric}$-a.e.\ such that $\Ric=\omega\abs{\Ric}$, in the sense that for every $X,Y\in \WHHSs$ we have that $X\otimes Y\,\cdot\,\omega\in\Lp^1(\abs{\Ric})$ and it holds that, as measures,
 \begin{equation}\label{bigclaim2}
 X\otimes Y\,\cdot\,\omega\abs{\Ric}=\Ric(X,Y).
 \end{equation}
Moreover
\begin{equation}\notag
\omega\abs{\Ric}\ge K \mathrm{g}\mass,
\end{equation}
in the sense that for every $v\in\tanXcap$, it holds, as measures,
\begin{equation}\label{ineqmeasr}
	v\otimes v\,\cdot\,\omega\abs{\Ric} \ge K\abs{v}^2{\mass}.
\end{equation}
Finally, as measures,
\begin{equation}\label{whyabs}
\abs{\Ric}=\sup_{X,Y}\abs{\Ric(X,Y)},
\end{equation}
where the supremum is taken among all vector fields $X,Y\in\WHHSs\cap\tanXinf$ such that $\Vert X\Vert_{\tanXinf}\le 1$ and $\Vert Y\Vert_{\tanXinf}\le 1$.
 \end{thm}
\begin{proof}
We divide the proof in several steps.
\medskip
\\\textbf{Step 1.} Uniqueness follows as in the proof of Proposition \ref{uniquenesshess},  by a localized version of Lemma \ref{densitytestbargentens}.
\medskip
\\\textbf{Step 2.} We remark that for every $X,Y\in\WHHSs$, it holds  $|\Ric(X,Y)|\ll\capa$, as a an immediate consequence of \eqref{rictest} together with a density and continuity argument (notice that it is enough to show that $\Ric(X,Y)(K)=0$ whenever $K$ is a compact set with $\capa(K)=0$).
\medskip
\\\textbf{Step 3.} We proceed now as in \textbf{Step 5} of the proof of Theorem \ref{hessrepr}.  In particular, we use Theorem \ref{captangent} to take a partition  of $\XX$,  $\{A_k\}_k$. We fix for the moment $k$ and take, (following Theorem \ref{captangent}), an orthonormal basis of $\tanXcap$ on $A_k$, $v^k_1,\dots, v^k_{n(k)}$.
Define, for $i,j=1,\dots,n(k)$, $$\mu_{i,j}^k\defeq\Ric(v^k_i,v^k_j)\mres A_k\qquad\text{and}\qquad\mu^k\defeq \sum_{i,j=1}^{n(k)} |{\mu_{i,j}^k}|.$$ 

We define $$\tilde{\omega}^k\defeq \sum_{i,j=1}^{n(k)} v_i^k\otimes v_j^k\dv{\mu_{i,j}^k}{\mu^k}$$
then $$\omega^k\defeq\frac{1}{|{\tilde{\omega}^k}|} \tilde{\omega}^k\qquad\text{and}\qquad|\Ric^k|\defeq {|{\tilde{\omega}^k}|}\mu^k.$$
Finally, $$\omega\defeq \sum_k\omega^k\qquad\text{and}\qquad\abs{\Ric}\defeq\sum_k|\Ric^k|.$$
Clearly, $\abs{\Ric}$ is a $\sigma$-finite measure, $\abs{\omega}=1\ \abs{\Ric}$-a.e.\ and $\omega$ is symmetric.
\medskip
\\\textbf{Step 4.} Let $X,Y\in\WHHSs$. We verify that $X\otimes Y\,\cdot\,\omega\in\Lp^1(\abs{\Ric})$ and that \eqref{bigclaim2} holds. This will be similar to \textbf{Step 6} of the proof of Theorem \ref{hessrepr} and we keep the same notation.
By polarization, there is no loss of generality in assuming that $X=Y$ and it is enough to show that for every $k$, as measures,
\begin{equation}\label{ambclaim1}
\Ric(X,X)\mres A_k= X\otimes X\,\cdot\, \omega\abs{\Ric}\mres A_k
\end{equation}
(notice that the left hand side of \eqref{ambclaim1} is a finite signed measure).

Assume for the moment  that also $X\in\tanXinf$.
Notice that \eqref{ambclaim1} also yields integrability (still in the case $X\in\tanXinf$), as it shows that
\begin{equation}\label{l1etv}
\Vert{X\otimes X\,\cdot\, \omega}\Vert_{\Lp^1(\abs{\Ric})}=\Vert \Ric(X,X) \Vert_{\TV}.
\end{equation}
Also, by \eqref{l1etv}, we see that the additional assumption $X\in\tanXinf$ is not restrictive: if $X\in\WHHSs$, we can find a sequence $\{X_n\}_n\subseteq \WHHSs\cap\tanXinf$ with $X_n\rightarrow X$ in the $\WSHs$ topology. For example, we can define  $$X_n\defeq\frac{n}{n\vee \abs{X}}X$$  $\WHHSs\cap\tanXinf\ni X_n\rightarrow X$ in the $\WSHs$ topology thanks to the calculus rules of Lemma \ref{calculushodge} and the  computation
$$
\abs{\nabla \left( \frac{n}{n\vee\abs{X} }\right)}\abs{X}= \chi_{\{\abs{X}\ge n\}} n \frac{\abs{\nabla \abs{X}}}{\abs{X}^2}\abs{X} \le \chi_{\{\abs{X}\ge n\}}\abs{ \nabla\abs{X}}\rightarrow 0\qquad\text{in }\Lpt.
$$
Then, by the continuity of $\Ric$ and \eqref{l1etv}, the sequence $\{X_n\otimes X_n\,\cdot\, \omega\}_n\subseteq{\Lp^1(\abs{\Ric})}$ is a Cauchy sequence, whose limit coincides then with $X\otimes X\,\cdot\, \omega$ so that this implies the general case.

We prove now \eqref{ambclaim1}, under the additional assumption $X\in\tanXinf$. 
We first remark that it holds  $$f\Ric(X,\,\cdot\,)=\Ric(f X,\,\cdot\,)\qquad\text{if }f\in\Ss\cap\Lpi.$$
This is a consequence of \cite[Proposition 3.6.9]{Gigli14} together with an approximation argument, see Lemma \ref{calculushodge} (here we use that $X\in\tanXinf$).  
Therefore, with the same computations of \textbf{Step 6} of the proof of Theorem \ref{hessrepr}, we see that 
$$X\otimes X\,\cdot\, \omega\abs{\Ric}\mres A_k=\Ric(\tilde{X},\tilde{X})\mres A_k $$
where $\tilde{X}\in\WHHSs\cap\tanXinf$ is such that $\tilde{X}=X\ \capa$-a.e.\ on $A_k$, so that  \eqref{ambclaim1} reduces to the locality relation 
\begin{equation}\notag
\Ric({X},{X})\mres A_k =\Ric(\tilde{X},\tilde{X})\mres A_k 
\end{equation}
whenever $X,\tilde{X}\in\WHHSs\cap\tanXinf$ are such that $\tilde{X}=X\ \capa$-a.e.\ on $A_k$.

By bilinearity, we can just show that if $X\in\WHHSs\cap\tanXinf$ is such that ${X}=0\ \capa$-a.e.\ on $A_k$, then $\Ric(X,Y)\mres A_k=0$ for every $Y\in\WHHSs\cap\tanXinf$.
Let $\{\varphi_n\}_n\subseteq\HSs\cap\Lpi$ be as in Lemma \ref{Amb} for the vector field $X\in\WHHSs\cap\tanXinf$ and let $h\in\LIPbs(\XX)$ (notice that $\varphi_n X\in\WHHSs$ by Lemma \ref{calculushodge}). 
We know that 
$$\int_\XX h\varphi_n\diff{\Ric(X,Y)}=\int_\XX h \diff{\Ric(\varphi_nX,Y)}.$$
By Lemma \ref{Amb} and the continuity of the map $\Ric$, the right hand side of the equation above converges to $0$, whereas the left hand side converges to (as in \textbf{Step 6} of the proof of Theorem \ref{hessrepr}) $$\int_\XX h\chi_{\{\abs{X}=0\}}\diff{\Ric(X,Y)}.$$
This is to say that for every $h\in\LIPbs(\XX)$ $$\int_\XX h\chi_{\{\abs{X}=0\}}\diff{\Ric(X,Y)}=0$$
which means that ${\Ric(X,Y)}\mres\{\abs{X}=0\}=0$, whence the claim.
\medskip
\\\textbf{Step 5.} Inequality \eqref{ineqmeasr}  follows by an approximation argument as in \textbf{Step 8} of the proof of Theorem \ref{hessrepr}, \eqref{bigclaim2} and \eqref{ricpos}.
\medskip
\\\textbf{Step 6.}
Equality \eqref{whyabs} follows by \eqref{bigclaim2} and a localized version of Lemma \ref{densitytestbargentens}.
\end{proof}
At the very end of \cite{Gigli14}, it has been asked how one may enlarge the domain of definition of the map $\Ric$, and, towards this extension, whether 
\begin{equation}\label{cdsasc}
	\sum_i f_i\Ric(X_i,Y)=\Ric\left( \sum_i f_i X_i,Y\right)
\end{equation}
whenever $X_1,\dots,X_n,Y\in\WHHSs$ and $f_i\in\Cb(\XX)$. It seems that basic algebraic manipulations based on the formulas involving the Ricci curvature as in Theorem \ref{GigRic} do not imply this  fact. However, exploiting  Theorem \ref{mainricc}, we immediately have an affirmative result to this question, at least in the finite dimensional case, and we record this result in the following proposition. More generally, Theorem \ref{mainricc} gives a natural way to enlarge the domain of definition of the map $\Ric$. 
\begin{prop}
	Let $(\XX,\dist,\mass)$ be an $\RCD(K,N)$ space. Let $X,Y,X_1,\dots,X_n\in\WHHSs$ and let $f_1,\dots,f_n\in \Cb(\XX)$ such that $X=\sum_{i=1}^n f_i X_i$. Then 
	$$
	\sum_{i=1}^n f_i\Ric(X_i,Y)=\Ric(X,Y).
	$$
\end{prop}
Notice also that an immediate consequence of Theorem \ref{mainricc}  (the point is \textbf{Step 4} of its proof, which relies on an approximation argument based on Lemma \ref{Amb}) is that for every $X,Y\in \WHHSs$,
$\Ric(X,Y)\mres \{|X|=0\}=0$, thus providing a different proof of the implication $3)\Rightarrow 1)$ of  \cite[Proposition 3.7]{Hannew} (see the comments at \cite[Pag. 3 and Pag. 4]{Hannew}).
\begin{rem}\label{remricc}
Exploiting to the representation of $\Ric=\omega|\Ric|$ given by Theorem \ref{mainricc}, we can easily give a meaning to the trace the \say{tensor} measure $\Ric$.
However we \emph{do not} expect that the trace of this polar measure  has a meaning to represent the scalar curvature, if one does not add artificial correction terms (cf.\ the characterization of the scalar curvature on smoothable Alexandrov spaces in \cite{lebedeva2022curvature}). 

First, already in the setting of a smooth weighted Riemannian manifold $(M,\dist_g,e^{-V}{\rm Vol}_g)$, $\Ric$ represents the modified Bakry-\'Emery $N$-Ricci curvature tensor, defined as
\begin{equation}
	\label{eq:beric}
	\mathrm{Ric}_N:=\left\{\begin{array}{ll}
		\mathrm{Ric}_g+ \hess_g(V) -\frac{\dd V\otimes \dd V}{N-n}&\qquad\text{ if }N>n,\\
		\mathrm{Ric}_g&\qquad\text{ if $N=n$ and $V$ is constant},\\
		-\infty&\qquad\text{ otherwise}.
	\end{array}\right.
\end{equation}
Nevertheless, even if we restrict ourselves to non-collapsed spaces (\cite{DPG17}), which play the role of \say{unweighted} spaces, we can see that looking at the scalar curvature as trace of $\Ric$ is not yet meaningful: for example $\Ric$ vanishes on sets of $0$ capacity, whereas the scalar curvature such behaviour is not expected  (e.g.\ if we want to have an analogue of Gauss Theorem for a two dimensional cone we have to allow the scalar curvature to recognize the singularity at the tip of the cone).
\fr
\end{rem}
\subsection{Riemann tensor}\label{sectionriem}
As done for Hessian and Ricci tensor, now we provide a representation for the Riemann curvature tensor defined in \cite{GigliRiemann}  as the product of a $\capa$-tensor field and a $\sigma$-finite measure that is absolutely continuous with respect to $\capa$. In order to do so, we again employ Riesz's representation Theorem for positive functional, and hence  we have to impose that the tensor representing the sectional curvature is bounded from below (hence, we will add the assumption of a bound on the distributional sectional curvature). Then, by standard algebra, we recover the full Riemann tensor out of the sectional curvatures.

\medskip

We follow \cite{GigliRiemann}  to define $\bnabla_X Y$, $\distrlie X Y $, $\ru X Y Z$ and $\rd X Y Z W$ on an $\RCD(K,\infty)$ space $(\XX,\dist,\mass)$. Even though we assume familiarity with \cite{GigliRiemann},  we recall briefly the (distributional) definitions. First, recall the definition of $\VV$ in \eqref{oldtest} . Then we have what follows.
\begin{enumerate}
	\item \textbf{Distributional covariant derivative.} If $X,Y\in\tanX$ with $X\in D(\dive)$ and at least one of $X$, $Y$ is in $\tanXinf$, then 	
\begin{equation}\notag
		\bnabla_X Y(W)\defeq -\int_\XX \nabla_X W\,\cdot\, Y+Y\,\cdot\,W \dive X\dd\mass\qquad\text{for every }W\in\VV.
\end{equation}
	
	\item \textbf{Distributional Lie bracket.} If $X,Y\in D(\dive)\cap\tanXinf$,
\begin{equation}\notag
	\distrlie{ X}{ Y}\defeq 	\bnabla_X Y-	\bnabla_Y X.
\end{equation}
	\item  \textbf{Distributional curvature tensor.} If $X,Y,Z\in\VV$, then 
\begin{equation}\notag
	\ru X Y Z\defeq \bnabla_X(\nabla_Y Z)- \bnabla_Y(\nabla_X Z)-\nabla_{[X,Y]} Z.
\end{equation}
	\item \textbf{Distributional Riemann tensor.} If $X,Y,Z,W\in\VV$, then
\begin{equation}\notag
	\rd X Y Z W (f)\defeq (	\ru X Y Z)(f W)\qquad\text{for every }f\in\TestF(\XX).
\end{equation}
\end{enumerate}
It is clear that the distributional covariant derivative and the distributional Lie bracket  coincide with the covariant derivative $\nabla_X Y$ and the Lie bracket $[X,Y]\defeq \nabla_X Y-\nabla_Y X$,  whenever both the  objects are defined. We are going to exploit this property throughout.

\begin{rem}\label{usefulcomputationspre}
We want to extend the definition of $\rd X Y Z W$ to the case $X,Y,Z,W\in\TestV(\XX)$ do not necessarily belong to $\VV$. Clearly, as $\TestV(\XX)\subseteq D(\dive)\cap\tanXinf$, the first two terms $\bnabla_X(\nabla_Y Z)(f W)-\bnabla_Y(\nabla_X Z)(f W)$ are still well defined. Also the third term $\nabla_{[X,Y]} Z$ makes sense for this choice of vector fields. Notice that in \cite{GigliRiemann}, $\bnabla_{\distrlie{X}{Y}} Z$ was used instead of  $\nabla_{[X,Y]} Z$. This clearly makes no substantial difference, but allows us to drop the request $\distrlie{X}{Y}=[X,Y]\in D(\dive)$, which is not granted if $X,Y\in\TestV(\XX)$ do not necessarily belong to  $\VV$,  cf.\ \cite[Lemma 2.4]{GigliRiemann}. 

Following the same lines, we see that $\rd X Y Z W$ makes sense  for every $X,Y,Z,W\in\WHHSs\cap\tanXinf$ and then $f\in\Ss\cap\Lpi$.
\fr
\end{rem}
\begin{rem}\label{usefulcomputations}
We do some trivial algebraic manipulations in order to deal with  the quantity $\rd X Y Z W (f)$, for $X,Y,Z,W\in\WHHSs\cap\tanXinf$ and $f\in\Ss\cap\Lpi$.
We have that 
\begin{align*}
	\rd X Y Z W (f)&=-\int_\XX \nabla_X (f W) \,\cdot\,\nabla_Y Z+ \nabla_Y Z\,\cdot\,(f W)\dive(X)\dd\mass \\&\quad +\int_\XX \nabla_Y (f W)\,\cdot\,\nabla_X Z+ \nabla_X Z\,\cdot\,(f W)\dive(Y)\dd\mass \\&\quad +\int_\XX f \nabla Z (\nabla_X Y-\nabla_Y X, W)\\&=
	-\int_\XX f  \nabla W (X,\nabla_Y Z) +\nabla f\,\cdot\, X \nabla Z(Y,W)+ f \nabla Z(Y,W)\dive(X)\dd\mass \\&\quad +\int_\XX f \nabla  W(Y,\nabla_X Z)+\nabla f\,\cdot\, Y \nabla Z(X,W)+ f\nabla Z( X,W)\dive(Y)\dd\mass \\&\quad +\int_\XX f \nabla Z (\nabla_X Y, W)-\int_\XX f \nabla Z (\nabla_Y X, W)\dd\mass. 
	\tag*{\fr}
\end{align*}
\end{rem}
In the sequel, we will \textbf{tacitly extend the definition} of ${\rd X Y Z W}$ to the case $X,Y,Z,W\in\WHHSs\cap\tanXinf$ and $f\in\Ss\cap\Lpi$, according to Remark \ref{usefulcomputationspre} and Remark \ref{usefulcomputations}.

For future reference, we recall here \cite[Proposition 2.7]{GigliRiemann}. Notice that an immediate approximation argument (recall Remark \ref{usefulcomputations}) allows us to extend the claim to the slightly larger class of vector fields and functions that we are considering. 
\begin{prop}[Symmetries of the curvature]\label{symms}
For any $X,Y,Z,W\in\WHHSs\cap\tanXinf$ and $f\in\Ss\cap\Lpi$ it holds:
	\begin{align*}
		\rd XYZW&=-\rd YXZW=\rd ZWXY,\\
		\ru XYZ&+\ru YZX+\ru ZXY=0,\\
		f\rd XYZW&=\rd{fX}{Y}{Z}{W}=\rd{X}{fY}{Z}{W}=\rd{X}{Y}{fZ}{W}=\rd{X}{Y}{Z}{fW}.
	\end{align*}
\end{prop}
\

The following definition has been implicitly proposed in \cite[Conjecture 1.1]{GigliRiemann}.
\begin{defn}\label{csscdas}
Let $(\XX,\dist,\mass)$ be an $\RCD(K,\infty)$ space. We say that $(\XX,\dist,\mass)$ has sectional curvature bounded below by $\kappa$, for some $\kappa\in\RR$, if for every $X,Y\in\TestV(\XX)$ and $f\in\TestF(\XX)$, $f\ge 0$, it holds
\begin{equation}\notag
	\rd X Y Y X(f)\ge \kappa \int_\XX f |X\wedge Y|^2\dd\mass.
\end{equation}
\end{defn}
\begin{rem}
It would be interesting to analyse the links between sectional curvature bounds in the sense of Definition \ref{csscdas} and in the sense of Alexandrov. This question is the content of  \cite[Conjecture 1.1]{GigliRiemann}.
\fr
\end{rem}

\begin{rem}\label{polarization}
It is well known that sectional curvatures (i.e.\ $\rd X Y Y X$) are sufficient to identify a unique full Riemann curvature tensor $\rd X Y Z W$. We write here an explicit expression, as we are going to need it in the sequel.
For $X,Y\in\WHHSs\cap\tanXinf$, set ${\bm{\mathcal K}} (X,Y)\defeq\rd X Y Y X$. We claim that
\begin{align*}
6	\rd X Y Z W&= 
	{\bm{\mathcal {\bm{\mathcal K}}}}(X+W,Y+Z)
	-{\bm{\mathcal {\bm{\mathcal K}}}}(X+W,Y)
	-{\bm{\mathcal {\bm{\mathcal K}}}}(X+W,Z)
	-{\bm{\mathcal {\bm{\mathcal K}}}}(Y+Z,X)\\&\qquad
	-{\bm{\mathcal K}}(Y+Z,W)
	+{\bm{\mathcal K}}(X,Z)
	+{\bm{\mathcal K}}(W,Y)
	-{\bm{\mathcal K}}(Y+W,X+Z)\\&\qquad
	+{\bm{\mathcal K}}(Y+W,X)
	+{\bm{\mathcal K}}(Y+W,Z)
	+{\bm{\mathcal K}}(X+Z,Y)
	+{\bm{\mathcal K}}(X+Z,W)\\&\qquad
	-{\bm{\mathcal K}}(Y,Z)
	-{\bm{\mathcal K}}(W,X).
\end{align*}
The claim follows by algebraic manipulation, by Proposition \ref{symms}.
See e.g.\ \cite[Lemma 4.3.3]{jost2008riemannian} for the expression.\fr
\end{rem}

\begin{thm}\label{riemrepr}
Let $(\XX,\dist,\mass)$ be an $\RCD(K,N)$ space with sectional curvature bounded below by $\kappa$, for some $\kappa\in\RR$. Then there exists a unique $\sigma$-finite measure $|\mathbf{Riem}|$ that satisfies $|\mathbf{Riem}|\ll\capa$ and a unique, up to $|\mathbf{Riem}|$-a.e.\ equality, tensor field $\nu\in \mathrm{L}^0_\capa(T^{\otimes 4}\XX)$ with $|\nu|=1\ |\mathbf{Riem}|$-a.e.\ such that for every $X,Y,Z,W\in \WHHSs\cap\tanXinf$ we have that $X\otimes Y\otimes Z\otimes W\,\cdot\,\nu\in \mathrm{L}^1(|\mathbf{Riem}|)$ and it holds
\begin{equation}\label{reprriemm}
	\int_\XX f X\otimes Y\otimes Z\otimes W\,\cdot\, \nu |\mathbf{Riem}|=\rd X Y Z W(f)\qquad\text{for every $f\in\Ss\cap\Lpi$}.
\end{equation}
For every $v,w\in\tanXcap$, $$v\otimes w\otimes w\otimes v\,\cdot\,\nu\ge \kappa |v\wedge w|^2.$$

The tensor field $\nu$ has the following symmetries. Let $\mathcal{I},\mathcal J,\mathcal{K}:\mathrm{L}^0_\capa(T^{\otimes 4}\XX)\rightarrow\mathrm{L}^0_\capa(T^{\otimes 4}\XX)$ be characterized as follows
\begin{align*}
	\mathcal I(v_1\otimes v_2\otimes v_3\otimes v_4)&\defeq v_2\otimes v_1\otimes v_3\otimes v_4\\
	\mathcal J(v_1\otimes v_2\otimes v_3\otimes v_4)&\defeq v_3\otimes v_4\otimes v_1\otimes v_2\\
 	\mathcal K(v_1\otimes v_2\otimes v_3\otimes v_4)&\defeq v_2\otimes v_3\otimes v_1\otimes v_4.
\end{align*}
Then, with respect to $|\mathbf{Riem}|$-a.e.\ equality,
\begin{align*}
\mathcal	I(\nu)&=-\nu\\
	\mathcal J(\nu)&=\nu\\
	\nu+\mathcal K(\nu)+\mathcal K^2(\nu)&=0.
\end{align*}
\end{thm}
\begin{proof}We divide the proof in several steps.
\medskip
\\\textbf{Step 1.} Uniqueness follows as in the proof of Proposition \ref{uniquenesshess},  by a localized version of Lemma \ref{densitytestbargentens}.
\medskip
\\\textbf{Step 2.} Notice that if $X,Y\in\WHHSs\cap\tanXinf$ and $f\in\Ss\cap\Lpi$, $f\ge 0$
then  $$\rd  X Y Y X(f)\ge \kappa\int_\XX f  | X\wedge  Y|^2\dd\mass,$$
thanks to an approximation argument that exploits the computations of Remark \ref{usefulcomputations} and \cite[Lemma 3.2]{BGBV}.

Therefore, Riesz-–Daniell–-Stone  Theorem yields that for every $X,Y\in\WHHSs\cap\tanXinf $ there exists a unique Radon measure $\mu_{X,Y}$ such that 
$$
\rd  X Y Y X(f)=\int_\XX f\dd\mu_{X ,Y}\qquad\text{for every }f\in\LIPbs(\XX).
$$
Clearly, for every $X,Y\in\WHHSs \cap\tanXinf$, $\mu_{X,Y}\ge \kappa |X\wedge Y|^2\dd\mass$ and the assignment $\WHHSs\cap\tanXinf\ni(X,Y)\mapsto\mu_{X,Y}$ is symmetric, by the symmetries of $\bm{\mathcal R}$ (Proposition \ref{symms}). Also,
we can prove, following \textbf{Step 2} of the proof of Theorem \ref{hessrepr} that $\mu_{X,Y}\ll\capa$, so that (see \textbf{Step 2} of the proof of Theorem \ref{hessrepr}  again and use an approximation argument based on the computations of Remark \ref{usefulcomputations})
$$
\rd  X Y Y X (f)=\int_\XX f\dd\mu_{X,Y}\qquad\text{for every }f\in\Ss\cap\Lpi,
$$
where we implicitly take the quasi continuous representative of $f$. This expression, together with the positivity of $\mu_{X,Y}-\kappa |X\wedge Y|^2$ yields that $\mu_{X,Y}$ is indeed a finite measure.
\medskip
\\\textbf{Step 3.} We notice that by Remark \ref{polarization} and by \textbf{Step 2}, if $X,Y,Z,W\in\WHHSs\cap\tanXinf$ then the map $\LIPbs(\XX)\ni f\mapsto \rd X Y Z W(f)$ is induced by a finite measure  $\mu_{X,Y,Z,W}$,
i.e.\ 
$$
\rd X Y Z W (f)=\int_\XX f\dd\mu_{X, Y, Z ,W}\qquad\text{for every }f\in\LIPbs(\XX),
$$
where
\begin{align*}
	\mu_{X, Y, Z ,W}&=\frac{1}{6}\Big(
	\mu_{X+W,Y+Z}
	-\mu_{X+W,Y}
	-\mu_{X+W,Z}
	-\mu_{Y+Z,X}
	-\mu_{Y+Z,W}
	+\mu_{X,Z}
	+\mu_{W,Y}\\&\qquad\qquad
	-\mu_{Y+W,X+Z}
	+\mu_{Y+W,X}
	+\mu_{Y+W,Z}
	+\mu_{X+Z,Y}
	+\mu_{X+Z,W}
	-\mu_{Y,Z}
	-\mu_{W,X}\big).
\end{align*}

Clearly, still $\mu_{X,Y,Z,W}\ll\capa$, hence, the equations above continue to hold even if only $f\in\Ss\cap\Lpi$.
Also, for $f\in\Ss\cap\Lpi$, and $X,Y,Z,W\in\WHHSs\cap\tanXinf$ 
\begin{equation}\label{csdma}
	f\mu_{X,Y,Z,W}=\mu_{f X,Y,Z,W}=\mu_{ X,f Y,Z,W}=\mu_{ X, Y,fZ,W}=\mu_{ X, Y,Z,fW}
\end{equation}
 by Proposition \ref{symms}.
 \medskip
\\\textbf{Step 4.} This is similar to \textbf{Step 5} of the proof of Theorem \ref{hessrepr}, we use the same notation. Let then $\{A_k\}$ and $\{v_i^k\}$ be as in \textbf{Step 5} of the proof of Theorem \ref{hessrepr}, building upon Theorem \ref{captangent}. Fix for the moment $k$ and
define, for $i,j,l,m=1,\dots,n(k)$,
$$
\mu_{i,j,l,m}^k\defeq\mu_{v_i^k+v_j^k,v_l^k+v_m^k}\mres A_k\qquad\text{and}\qquad  \mu_{i,j,l}^k\defeq\mu_{v_i^k+v_j^k,v_l^k}\mres A_k\qquad\text{and}\qquad \mu_{i,j}^k\defeq\mu_{v_i^k,v_j^k}\mres A_k
$$
and also  $$\mu^k\defeq\sum_{i,j,l,m=1}^{n(k)}\big(|\mu^k_{i,j,l,m}|+|\mu^k_{i,j,l}|+|\mu^k_{i,j}|\big).$$
Now we define
\begin{align*}
	\tilde\rho^k_{i,j,l,m}\defeq\dv{\mu_{v_i^k,v_j^k,v_l^k,v_m^k}\mres A_k}{\mu^k},
\end{align*}
notice that $\mu_{{v_i^k,v_j^k,v_l^k,v_m^k}}\ll\mu^k$ by construction for every $i,j,l,m=1,\dots n(k)$.
Set also
$$
\tilde\nu^k\defeq\sum_{i,j,l,m=1}^{n(k)} v^k_i\otimes v_j^k\otimes v^k_l\otimes v^k_m \tilde\rho^k_{i,j,l,m}.
$$
and 
$$
\nu^k\defeq\frac{1}{|\tilde\nu^k|}\tilde\nu^k\qquad\text{and}\qquad|\mathbf{Riem}^k|\defeq|\tilde\nu^k|\mu^k.
$$
Finally
$$
\nu\defeq\sum_k\nu^k\qquad\text{and}\qquad |\mathbf{Riem}|\defeq\sum_k|\mathbf{Riem}^k|.
$$

Clearly, $|\mathbf{Riem}|$ is a $\sigma$-finite measure, $|\mathbf{Riem}|\ll\capa$ and $|\nu|=1\ |\mathbf{Riem}|$-a.e.
\medskip
\\\textbf{Step 5}.
We claim that $$X\otimes Y\otimes Z\otimes W\,\cdot\,\nu^k\dd|\mathbf{Riem}^k|=\mu_{X,Y,Z,W}\mres A_k$$ for every $X,Y,Z,W\in\WHHSs\cap\tanXinf$ and for every $k$. Recall that $\int_\XX f\dd\mu_{X,Y,Z,W}=\rd X Y Z W (f)$ for every $f\in\Ss\cap\Lpi$, so that the claim will imply \eqref{reprriemm} and also the fact that  $$|X\otimes Y\otimes Z\otimes W\,\cdot\,\nu^k|\dd|\mathbf{Riem}^k|=|\mu_{X,Y,Z,W}|\mres A_k,$$
so that, being $\mu_{X,Y,Z,W}$ a finite measure, $X\otimes Y\otimes Z\otimes W \,\cdot\,\nu^k\in \mathrm{L}^1(|\mathbf{Riem}|)$.

Fix $k$ and  take then $X,Y,Z,W\in\WHHSs\cap\tanXinf$. We write $\tilde X\defeq \sum_{i=1}^{n(k)} X^i v_i^k$, for $X^i\defeq X\,\cdot\, v_i^k$ and similarly for $ Y, W, Z $. Notice  $X^i,Y^i,Z^i,W^i\in\HSs\cap\Lpi$ and  $\tilde X ,\tilde Y,\tilde Z,\tilde W \in\TestV(\XX)$. Notice that these newly defined functions and vector fields depend on $k$, but as we are working for a fixed $k$, we do not make this dependence explicit. We compute, on $A_k$,
\begin{align*}
	&X\otimes Y\otimes Z\otimes W\,\cdot\,\nu^k\dd|\mathbf{Riem}^k|=\sum_{i,j,l,m=1}^{n(k)}  X\otimes Y\otimes Z\otimes W\,\cdot v^k_i\otimes v^k_j\otimes v^k_l\otimes v^k_m \tilde\rho^k_{i,j,l,m}\dd\mu^k\\&\qquad
	=\sum_{i,j,l,m=1}^{n(k)}  X\otimes Y\otimes Z\otimes W\,\cdot v^k_i\otimes v^k_j\otimes v^k_l\otimes v^k_m\dd \mu_{v_i^k,v_j^k,v_l^k,v_m^k}\mres A_k
	\\&\qquad =\sum_{i,j,l,m=1}^{n(k)} X^i Y^jZ^l W^m\dd\mu_{v_i^k,v_j^k,v_l^k,v_m^k}\mres A_k=\sum_{i,j,l,m=1}^{n(k)}\dd\mu_{X^i v_i^k,Y^jv_j^k, Z^l v_l^k,W^mv_m^k}\mres A_k\\&\qquad=\mu_{\tilde X,\tilde Y,\tilde Z,\tilde W}\mres A_k,
\end{align*}
where the next to last equality is due to \eqref{csdma}.
We verify now that $\mu_{\tilde X,\tilde Y,\tilde Z,\tilde W}\mres A_k=\mu_{ X, Y, Z, W}\mres A_k$, which will conclude the proof of the claim. This will be similar to \textbf{Step 6} of the proof of Theorem \ref{hessrepr}. 

By multi-linearity and Proposition \ref{symms}, it is enough to show that for every $X,Y,Z,W\in\WHHSs\cap\tanXinf$, then $\mu_{X,Y,Z,W}\mres\{ |X|=0\}=0$.
We take $\{\varphi_n\}_n$  be as in Lemma \ref{Amb} for the vector field $X\in\WHHSs\cap\tanXinf$. We take also $h\in\LIPbs(\XX)$ and we compute (recall Lemma \ref{calculushodge})
$$
\int_\XX h\varphi_n \dd\mu_{X,Y,Z,W}=\int_\XX h\dd\mu_{\varphi_n X,Y,Z,W}=\rd {\varphi_n X} Y Z W (h).
$$
By Lemma \ref{Amb}  and the expression for the map $\bm{\mathcal R}$ in Remark \ref{usefulcomputations}, the right hand side of the equation above converges to $0$, whereas the left hand side converges to $$\int_\XX h\chi_{\{\abs{X}=0\}}\dd\mu_{X,Y,Z,W}.$$
This is to say that for every $h\in\LIPbs(\XX)$ $$\int_\XX h\chi_{\{\abs{X}=0\}}\dd\mu_{X,Y,Z,W}=0$$ whence the claim.
\medskip
\\\textbf{Step 6.} By approximation (Lemma \ref{densitytestbargentens}), it is enough to show the claim  for $X,Y\in\TestV(\XX)$. Then the claim follows from \eqref{reprriemm} and the assumption on the bound from below for the sectional curvature.
\medskip
\\\textbf{Step 7.} The symmetries claimed follow from Proposition \ref{symms}. We prove, for example, the first one. It is enough to show that $\nu+\mathcal I(\nu)=0$ with respect to $|\mathbf{Riem}|$-a.e.\ equality. Now, if $X,Y,Z,W\in\WHHSs\cap\tanXinf$, then 

\begin{align*}
	&\int_\XX X\otimes Y\otimes Z\otimes W\,\cdot\, (\nu+\mathcal I (\nu)) |\mathbf{Riem}|\\&\qquad=\int_\XX X\otimes Y\otimes Z\otimes W\,\cdot\, \nu |\mathbf{Riem}|+\int_\XX \mathcal I(X\otimes Y\otimes Z\otimes W)\,\cdot\, \nu |\mathbf{Riem}|\\&\qquad=\rd X Y Z W(1)+\rd Y  X Z W(1)=0,
\end{align*}
so that the claim follows from Lemma \ref{densitytestbargentens}.
\end{proof}

\begin{rem}
Notice that, thanks to its symmetries, the tensor field $\nu$ of Theorem \ref{riemrepr} can be seen as an element of $\big(\mathrm{L}^0_\capa(T\XX)^{\wedge 2}\big)^{\otimes 2}$.\fr
\end{rem}
\begin{rem}
Comparing the main results of Section \ref{sectionric} and Section \ref{sectionriem}, we may wonder whether $\Ric$ is linked to the trace of $\bm{\mathcal R}$. By what remarked in Remark \ref{remricc}, we see that this question makes sense only on non-collapsed spaces. However, the non-smooth structure of the space, in particular, the lack of a third order calculus and charts defined on open sets, prevent us to give an easy proof of this fact.\fr
\end{rem}

\end{document}